\documentclass[a4paper,11pt,english]{amsart}

\usepackage[english]{babel}
\usepackage[latin1]{inputenc}
\usepackage[T1]{fontenc}

\usepackage{graphicx}
\usepackage{amssymb}
\usepackage{verbatim,enumerate}

\newtheorem{thm}{Theorem}
\newtheorem{lemma}[thm]{Lemma}
\newtheorem{prop}[thm]{Proposition}
\newtheorem{cor}[thm]{Corollary}
\newtheorem{defi}[thm]{Definition}
{\theoremstyle{definition}
\newtheorem*{exa}{Example}
\newtheorem*{rem}{Remark}}
\newcommand{\C}{\mathbb{C}}
\newcommand{\CC}{\mathcal{C}}

\newcommand{\F}{\mathbb{F}}

\newcommand{\R}{\mathbb{R}}
\newcommand{\Z}{\mathbb{Z}}

\newcommand{\D}{\mathcal{D}}
\newcommand{\x}{\underline{x}}

\renewcommand{\F}{\mathcal F}

\begin{document}

\title{On Welschinger invariants of symplectic 4-manifolds}

\author{Erwan Brugall\'e}
\address{Erwan Brugall\'e, École Polytechnique,
Centre Mathématiques Laurent Schwartz, 91 128 Palaiseau Cedex, France}
\email{erwan.brugalle@math.cnrs.fr}

\author{Nicolas Puignau}
\address{Nicolas Puignau, Universidade Federal do Rio de Janeiro, Ilha do Fundão, 21941-909 Rio de Janeiro, Brasil}

\email{puignau@im.ufrj.br}

\subjclass[2010]{Primary 14P05, 14N10; Secondary 14N35, 14P25}
\keywords{Real enumerative geometry, Welschinger invariants,
Gromov-Witten invariants, symplectic sum formula, symplectic field theory}

\begin{abstract}
We prove the vanishing of
 many  Welschinger invariants of real symplectic $4$-manifolds.  
In some particular instances, we also determine their sign  and
show that they are divisible by a large power of $2$.
Those results are a consequence of 
several relations among Welschinger invariants
 obtained by a real version of symplectic sum formula.
In particular, this note contains   proofs of results announced in \cite{BP}.
\end{abstract}
\maketitle

A \textit{real symplectic manifold} $(X,\omega,\tau)$ is a symplectic
manifold $(X,\omega)$ equipped with an antisymplectic involution
$\tau$. The \textit{real part} of $(X,\omega,\tau)$, denoted by $\R X$, is by
definition the fixed point set of $\tau$.
We say that 
 an  almost complex structure $J$ tamed by $\omega$ is
 $\tau$-compatible if $\tau$ is $J$-antiholomorphic, i.e. $J\circ
 d \tau=-d\tau\circ J$.

\medskip
Let $X_\R=(X,\omega,\tau)$ be a real symplectic manifold of dimension
4. Let  $C$ be an immersed real rational 
$J$-holomorphic curve in $X$ for some  $\tau$-compatible almost
complex structure $J$, and denote by $L$ the  connected component 
of $\R X$ containing the 1-dimensional part $\widetilde{\R C}$ of $\R C$.
Fix also 
a $\tau$-invariant class $F$ in $H_2(X\setminus L;\Z/2\Z)$. 
Any half of $C\setminus \widetilde{\R C}$ defines 
a class $\underline C$ in $H_2(X,L;\Z/2\Z)$ whose intersection number
modulo 2 with $F$, denoted by $\underline C\cdot F$, 
 is well defined and
does not depend on 
the chosen half. 
 We further denote by $m(C)$ the number of nodes of  $C$ in $L$ with two
 $\tau$-conjugated branches, and
we
define the $F$-mass of $C$ as
$$m_{L,F}(C)=m(C)+\underline C\cdot F. $$

Choose  a connected component  $L$ of $\R X$, a class $d\in
H_2(X;\Z)$,
and  $r,s\in\Z_{\ge 0}$ such that  
$$c_1(X)\cdot d - 1= r+2s.$$
Choose a configuration $\x$ made of $r$ points in $L$ and  $s$ 
 pairs of $\tau$-conjugated 
points in $X\setminus \R X$.
Given a   $\tau$-compatible almost complex structure $J$,
we denote by $\mathcal C(d,\x,J)$
 the set of real rational
$J$-holomorphic curves  in $X$ realizing the class $d$, passing through $\x$,
 and such that $L$ contains  $\widetilde{\R C}$.
For a generic choice of $J$, the set  $\mathcal C(d,\x,J)$ is finite, and
the integer
$$W_{X_\R,L,F}(d,s)= \sum_{C\in\mathcal C(d,\x,J)}(-1)^{m_{L,F}(C)} $$
 depends neither on $\x$, $J$, nor on the deformation class of
$X_\R$ (see \cite{Wel1,IKS11})\footnote{Welschinger originally considered in \cite{Wel1} only 
the case when $F=[ \R X\setminus L]$. In this case $m_{L,F}(C)$ is
the number of solitary nodes of  $\R C$.
Later, Itenberg, Kharlamov, and Shustin observed in \cite{IKS11} that
Welschinger's proof extends literally to arbitrary $\tau$-invariant
classes
 in $H_2(X\setminus L;\Z/2\Z)$. See also \cite{Geor13} for a related
 discussion. 
\\ Note that our convention differ slightly from \cite{IKS11}, 
where  the sign of a curve in $\mathcal
C(d,\x,J)$  depends on the parity of
$ m(C)+\underline C\cdot (F+ [\R X\setminus L])$ instead 
of 
$ m(C)+\underline C\cdot (F)$.}. 
We call these numbers the \emph{Welschinger invariants of $X_\R$}.
When $F=[\R X\setminus L]$, we simply denote
$W_{X_\R,L}(d,s)$ instead of $W_{X_\R,L,[\R X\setminus L]}(d,s)$.
Note that Welschinger invariants are non-trivial to compute
 only in the case of  rational manifolds.

\medskip
A \emph{real Lagrangian sphere} of $X_\R$
is a Lagrangian sphere globally invariant under $\tau$.
Two disjoint surfaces $S$ and $S'$ in $X$ are said to be \emph{connected by
a chain of real Lagrangian spheres} if there exists real Lagrangian spheres
$S_1,\ldots,S_k$ in $X$ such that 
 $S_i\cap S_{j}=\emptyset$
if $|i-j|\ge 2$,
and $S_i$ and $S_{i+1}$ intersect
transversely in a single point, as well as $S$ and $S_0$, and $S'$ and $S_k$.

The next two theorems are the main results of this note.

\begin{thm}\label{thm:chain spheres}
Let $X_\R$ be a real symplectic $4$-manifold, and 
suppose that  $F$ has a $\tau$-invariant representative  connected to $L$ by a
chain of real Lagrangian spheres.
\begin{enumerate}
\item If $r\ge 2$, then
 $$W_{X_\R,L,F}(d,s)= 0.$$
\item If $r=1$ and $c_1(X)\cdot d\ge 2$, then 
$$2^{\frac{c_1(X)\cdot d  -4}{2}}\ |\ W_{X_\R,L,F}(d,s).$$
If in addition $F=[\R X\setminus L]$, then
$$ (-1)^{\frac{ d^2  -c_1(X)\cdot d    +2}{2}}W_{X_\R,L}(d,s)\ge
  0.$$
\end{enumerate}
\end{thm}
Theorem \ref{thm:chain spheres}
 is an immediate consequence of 
Theorem \ref{thm:W vanish} and Corollary \ref{cor:W equal} 
respectively given in Sections \ref{sec:W vanish} and  \ref{sec:relations}.
The invariant $W_{X_\R,L,0}(d,s)$ does not
seem to satisfy a vanishing statement analogous to 
Theorem \ref{thm:chain spheres}$(1)$
(see {\cite{IKS11,IKS13,Bru14}}), implying that the set  
$\mathcal C(d,\x,J)$ is usually
non-empty.
Theorem \ref{thm:chain spheres}$(2)$ 
 partially generalizes {\cite[Theorems 1.1,
    2.1, 2.2, and 2.3]{Wel4}} and 
 {\cite[Proposition 8.2]{Bru14}}.

\medskip
Theorem \ref{thm:chain spheres} can be specialized to real algebraic
rational surfaces, whose classification is well known (see \cite{Sil89,Kol97} for example).
A real algebraic rational surface is always implicitly assumed to be equipped
with some Kähler form.

Let $\mathcal G$ be the subgroup of the $\tau$-invariant classes in
$H_2(X\setminus L;\Z/2\Z)$ generated by the kernel of the natural map
$H_2(X\setminus L;\Z/2\Z)\to H_2(X;\Z/2\Z)$, and by the classes
realized by 
smooth
real
symplectic curves with either positive genus or  self-intersection at least $-1$.
We show in Propositions \ref{prop:curve} and \ref{prop:vanish} that 
$W_{X_\R,L,F}$ and $W_{X_\R,L,F'}$
are equal in absolute value if $F-F'\in\mathcal G$. 
We denote by $\mathcal H(X_\R,L)$ the
group
 of $\tau$-invariant classes in
$H_2(X\setminus L;\Z/2\Z)$ quotiented by $\mathcal G$.
All groups  $\mathcal H(X_\R,L)$ are computed in the case of 
real algebraic rational surfaces in Section \ref{sec:real rational}.
 In particular, we prove in 
Proposition \ref{prop:isomorphism H} that
they only depend on a minimal model of $X_\R$ and on the choice of $L$.

\begin{thm}\label{thm:real algebraic}
Let $X_\R$ be a real symplectic $4$-manifold  equal,   up to
deformation and equivariant symplectomorphism, to a
real algebraic rational surface, and suppose that $F$ is non-zero
in $\mathcal H(X_\R,L)$.
Then the conclusions of Theorem \ref{thm:chain spheres} hold in the
following cases:
\begin{itemize}
\item $X_\R$ is obtained from a minimal model by blowing up pairs of
  complex conjugated points and real points on at most two connected
  components of $\R X$, one of them being $L$;
\item $X_\R$ is a Del Pezzo surface;
\item $F=[\R X\setminus L]$. 
\end{itemize}
\end{thm}
\begin{rem}
In a burst of enthusiasm, we forgot in {\cite[Proposition 3.3]{BP}}
the assumption that $X_\R$ has
to be symplectomorphic/deformation equivalent to  
a   
real algebraic rational surface.
\end{rem}

\medskip
Theorem \ref{thm:real algebraic} follows from
the classification of real
algebraic rational surfaces and 
Theorem \ref{thm:chain spheres}, which in its turn
 is a
 direct consequence of Theorem \ref{thm:W vanish} 
and Corollary \ref{cor:W equal} below.
Our strategy to prove these latters is to degenerate $X_\R$ into a
reducible real symplectic manifold $X_{\sharp,\R}$, and to relate enumeration of
curves in $X_{\sharp,\R}$ and in $X_\R$. This degeneration can be thought as a
 degeneration of $X_\R$ to a real nodal 
symplectic manifold, and can be described by the contraction of a 
 real Lagrangian sphere $S_V$ by
stretching 
the neck of a $\tau$-compatible  almost complex structure  in a
neighborhood of $S_V$ (see \cite{EGH,Wel4}). In this note we use an
equivalent description in terms of symplectic sum (\cite{Gom95,IP}), see
section \ref{sec:degeneration} for more details. 

In particular,  Corollary \ref{cor:W equal} follows from 
Theorem  \ref{mainthm}, which can be seen as a real version
of the Abramovich-Bertram-Vakil formula {\cite[Theorem 3.1.1]{AB}},  
{\cite[Theorem 4.5]{Vak2}}.
Another but related treatment of  contraction of Lagrangian spheres contained
in $\R X$ has  previously been proposed by Welschinger in
\cite{Wel4}.

\medskip
The paper is organized as follows. We state Theorems \ref{thm:W
  vanish} and \ref{mainthm} in Section \ref{sec:formula}, and give
  their proof in  Section \ref{sec:proofs} using a
  real version of the symplectic sum 
formula.
We
end this paper by explicit computations in the case of real algebraic
rational surfaces in Section \ref{sec:real rational}.

\bigskip
\textbf{Acknowledgment: } We are grateful to
 Simone Diverio, Penka Georgieva, Umberto Hryniewicz, Ilia Itenberg,
Viatcheslav Kharlamov, Leonardo Macarini, Frédéric Mangolte, Brett Parker, Christian
Peskine, Patrick Popescu,  Jean-Yves Welschinger, and
Aleksey Zinger for many useful conversations.
We are also indebted to the anonymous referee for many valuable
comments on the first version of this paper.

Both authors were supported by the Brazilian-French Network in Mathematics.  E.B. was also partially supported by the ANR-09-BLAN-0039-01. 
Part of this work was accomplished at the Centre Interfacultaire
Bernoulli (CIB) in Lausanne, Switzerland, during the semester program
 ``Tropical geometry in its complex and symplectic aspects''.

\section{Auxiliary results}\label{sec:formula}
\subsection{Preliminaries}
In the whole text, we denote by $X_0=\C P^1\times \C P^1$, by $\omega_{FS}$ the
  Fubini-Study form on $\C P^n$, and by $l_1$ and $l_2$ respectively
  the homology classes $[\C P^1\times \{0\}]$ and $[ \{0\}\times\C
    P^1]$ in $H_2( X_0;\Z)$. Recall that $H_2( X_0;\Z)$ is the
  free abelian group generated by $l_1$ and $l_2$.
Up
to conjugation by an automorphism, there exist
four different real structures on 
$(\C P^1\times \C P^1, \omega_{FS}\oplus \omega_{FS})$, 
and  the class $l_1+l_2$ is invariant for
exactly three of them, see for example \cite{Sil89,Kol97}. 
These latter are given in coordinate by:
\begin{itemize}
\item $\tau_{hy}(z,w)=(\overline z,\overline w)$, $\R X_{hy}=S^1\times
  S^1$;
\item  $\tau_{el}(z,w)=(\overline w,\overline z)$, $\R X_{el}=S^2$;
\item  $\tau_{em}(z,w)=(-\frac{1}{\overline z},-\frac{1}{\overline w})$,
  $\R X_{em}=\emptyset $.
\end{itemize}

Note that $\tau_{hy}$ and  $\tau_{em}$ act trivially on $H_2(X_0;\Z/2\Z)$,
while $\tau_{el}$ exchanges the classes $l_1$ and $l_2$. Note also,
with the convention that $\chi(\emptyset)=0$, that
$$ \chi(\R X_{hy})=  \chi(\R X_{em})=0, \quad \textrm{and}\quad 
\chi(\R X_{el})=2 .$$

\begin{lemma}
Let  $E$ be a 
smooth symplectic curve in $(X_0,\omega_{FS}\oplus\omega_{FS}) $ realizing the class $l_1+l_2$ in $H_2(X_0;\Z)$.
The group $H_2( X_0\setminus E;\Z/2\Z)$ is isomorphic to $\Z/2\Z$, and
is generated by any representative
 disjoint from $E$ of the
class $l_1+l_2$ in $H_2( X_0;\Z/2\Z)$.
\end{lemma}
\begin{proof}
The first Chern class of  $(X_0,\omega_{FS}\oplus\omega_{FS})$ is dual
to $l_1+l_2$. Hence
it follows from the adjunction formula \cite[Chapter 2]{McS} that $E$
  is an embedded sphere. 
The lemma can be proved exactly as Lemma \ref{lem:MV}, nevertheless we
provide an alternate proof. 

Let $J$ be an almost complex structure on $X_0$ tamed by
$\omega_{FS}\oplus\omega_{FS}$ such that $E$ is $J$-holomorphic.
Since both classes $l_1$ and $l_2$ have the same symplectic area, a
class $al_1+bl_2$ has positive symplectic area if and only if
$a+b>0$. As a consequence, any 
  $J$-holomorphic curve realizing the
class $l_1$ is an embedded sphere. The Gromov-Witten invariant of
$(X_0,\omega_{FS}\oplus\omega_{FS}) $ for the class
$l_1$ is equal to 1, and $l_1^2=0$, so  there exists a unique 
$J$-holomorphic sphere realizing the class $l_1$  and 
passing through any given point of $X_0$. Recall that any intersection
of two distinct
$J$-holomorphic curves is positive.
Since
$[E]\cdot l_1=1$, we deduce a $S^2$-fibration $X_0\to E$ whose fiber
over a point $p\in E$ is the  $J$-holomorphic sphere realizing the
class $l_1$  and  
passing through $p$. In its turn, this induces  a $\R^2$-fibration 
$X_0\setminus E\to E$, and so $X_0\setminus E$ has the same homotopy
type than $E$. This proves that $H_2(X_0\setminus E;\Z/2\Z)\simeq \Z/2\Z$.

Since $[E]^2=2$ in $X_0$, there exists a 
 representative $F$
of the
class $l_1+l_2$ in $H_2( X_0;\Z/2\Z)$
 disjoint from $E$. The class $[F]$ is obviously non-zero in $H_2(
X_0\setminus E;\Z/2\Z)$, and so generates the group.
\end{proof}

\begin{lemma}\label{lem:real disk}
Suppose
that $E$ is a smooth real symplectic curve in $(X_0,\omega_{FS}\oplus
\omega_{FS},\tau_{el}) $ realizing the class $l_1+l_2$ in
$H_2(X_0;\Z)$, and that $\D$ is an embedded
 $\tau_{el }$-invariant disk with  $\partial
\D\subset E$.
Then the group $H_2( X_0, E;\Z/2\Z)$ is isomorphic to $\Z/2\Z$ and
 generated by $\D$.
\end{lemma}
\begin{proof}
Recall that $E$ is an embedded sphere.
The  long exact sequence of pairs gives the exact sequence
$$H_2(E;\Z/2\Z)\xrightarrow[]{i} H_2(X_0;\Z/2\Z) \xrightarrow[]{j} 
H_2(X_0,E;\Z/2\Z) \xrightarrow[]{} 0. $$
The map $i$ is clearly injective, so $H_2(X_0,E;\Z/2\Z)$ is
isomorphic to $\Z/2\Z$ and generated by $j(l_1)=j(l_2)$.

Denote by $D_1$ and $D_2$ the two halves of $E\setminus \partial \D$. 
Since $E$ is a real symplectic curve, the involution  $\tau_{el }$
exchanges $D_1$ and $D_2$.
The surface $D_i\cup \D$  realizes a class in 
$H_2(X_0;\Z/2\Z)$, and we have
$$l_1+l_2=[D_1\cup D_2]= [D_1\cup \D] + [D_2\cup \D] \quad
\mbox{in}\ H_2(X_0;\Z/2\Z). $$ 
Since $\tau_{el}$ exchanges the classes $[D_1\cup \D]$ and $[D_2\cup
  \D] $, both of them are non-null, i.e. 
$[D_1\cup \D]=l_i$ and $[D_2\cup
  \D] =l_{3-i}$. Hence by the long exact sequence of pairs, the class
realized by $D_1\cup \D$ in $H_2( X_0, E;\Z/2\Z)$, which equals the
class realized by $\D$,  generates the
group.
\end{proof}

\begin{exa}
In the case when $X_0\setminus E$ is the affine 
quadric with equation $x^2+y^2+z^2=1$ in $\C^3$, the sphere
$X_0\cap \R^3$ is an example of
  generator of $H_2( X_0\setminus E;\Z/2\Z)$, and the disk 
$X_0\cap (i\R\times i\R\times \R_{>0})$ is an example of 
generator of $H_2( X_0, E;\Z/2\Z)$, see Figure \ref{fig:H2}c.
\end{exa}

\subsection{Vanishing Lagrangian spheres}\label{sec:degeneration}
Let $X_\R=(X,\omega,\tau)$ be a real symplectic manifold of dimension
4.
A class $V$  in $H_2(X;\Z/2\Z)$ is called a \emph{real vanishing cycle} if
it can be represented by a real Lagrangian sphere $S_V$.
By
stretching 
the neck of a $\tau$-compatible  almost complex structure  in a
neighborhood of $S_V$, one  decomposes $X$ into the union of $X\setminus
S_V$ and $T^*S_V$.
 This operation can be thought as a
 degeneration of $X_\R$ to a real nodal 
symplectic manifold for which $V$ is precisely the vanishing cycle. 
Equivalently, the class $V$ is a real vanishing cycle if and only if, up to
deformation, 
$X_\R$
can be represented as the real 
symplectic sum of two real symplectic manifolds $(X_1,\omega_1,\tau_1)$
and $(X_0,\omega_{FS}\oplus
\omega_{FS},\tau_0)$ along
an embedded symplectic  
sphere $E$ of self-intersection
$-2$ in $X_1$ (hence of self-intersection
$2$ in $X_0$) where:
\begin{itemize}
\item $E$ is real and realizes the class $l_1+l_2$ in $H_2(X_0;\Z)$;
\item $V$ is represented by the deformation in $X$ of a representative
  of the
non-trivial class in $H_2(X_0\setminus E;\Z/2\Z)$.
\end{itemize}
By abuse, we still denote by
$V$ the non-trivial class in $H_2(X_0\setminus E;\Z/2\Z)$.
We refer to Section \ref{sec:sympl sum} for more details about the
symplectic sum 
operation. 
We denote by $X_\sharp$ the union of $(X_1,\omega_1,\tau_1)$
and $(X_0,\omega_{FS}\oplus
\omega_{FS},\tau_0)$ along $E$,  by
$L_\sharp$ the degeneration of $L$ as $X_\R$ degenerates to
$X_\sharp$, and by  $L_i$ the intersection $L_\sharp\cap X_i$. Note
that by construction we have $\partial L_i\subset \R E$.

Recall that $T^*S^2$ is equivariantly 
symplectomorphic to the complement of a smooth
real hyperplane section $E$ of 
a smooth real quadric in
$\C P^3$.
This real quadric is precisely the
 summand $(X_0,\omega_{FS}\oplus \omega_{FS},\tau_0)$ of
 $X_\sharp$. We depicted in Figure \ref{fig:SV} all possibilities for
 $\R X_0$ and $S_V$.
Choose a  diffeomorphism $\Psi$ between $T^*S^2$ and 
 the line bundle 
$\mathcal O_{\C  P^1}(-2)$
 of degree $-2$
over $E=\C P^1$, which restricts to a
symplectomorphism 
between
the  complements of the zero sections
 (note that $\Psi$ does not preserve the fibration).  
The summand  $(X_1,\omega_1,\tau_1)$  of $X_\sharp$
is obtained by removing from
$X$ a small
tubular neighborhood of $S_V$, and by gluing back via $\Psi$ a small
neighborhood of the zero section of 
$\mathcal O_{\C  P^1}(-2)$.
The homology groups $H_2(X_1;\mathbb Z)$ and $H_2(X;\mathbb
Z)$ are canonically identified, the class  $[E]$ being
identified with the class $\Psi^{-1}_*([E])$.
We implicitly use this identification throughout the text.

\begin{figure}[h!]
\begin{center}
\begin{tabular}{ccccc}
\includegraphics[width=3cm, angle=0]{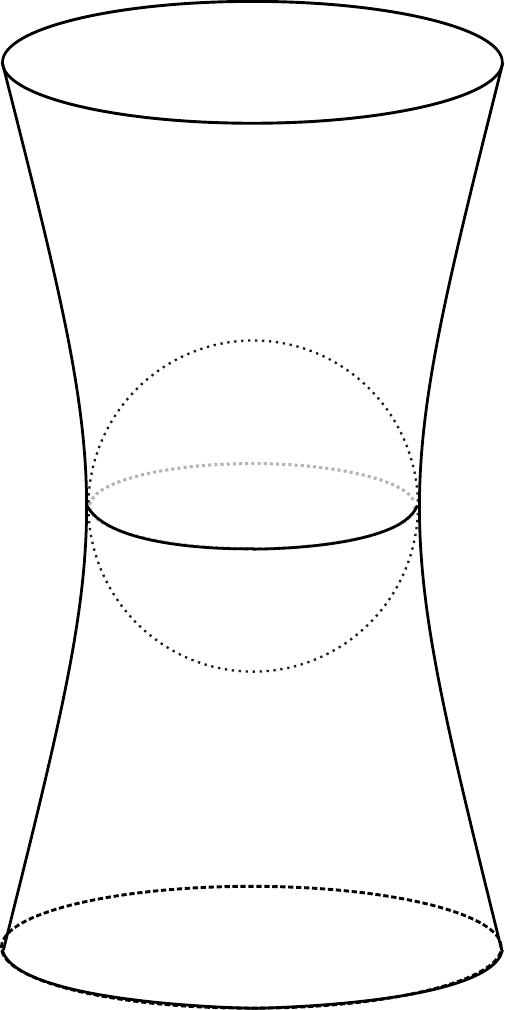}
\put(-110, 130){$\R X_0$}
\put(-50, 45){$S_V$}
& \hspace{3ex} &
\includegraphics[width=3cm, angle=0]{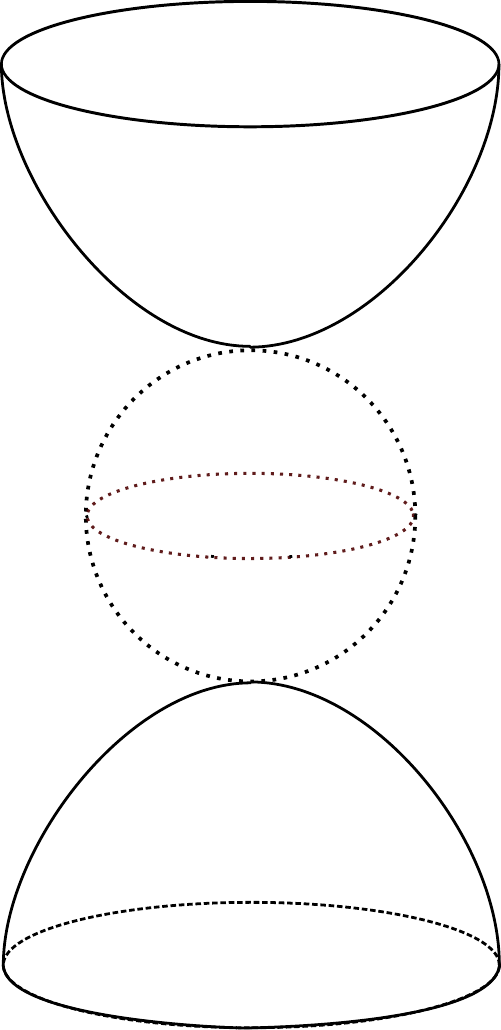}
\put(-110, 140){$\R X_0$}
\put(-110, 20){$\R X_0$}
\put(-10, 80){$S_V$}
& 
\includegraphics[width=3cm, angle=0]{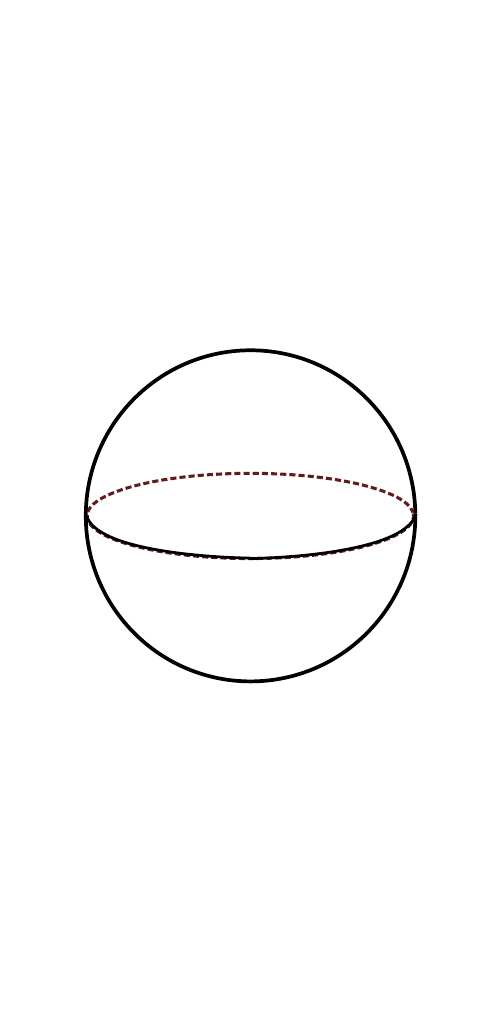}
\put(-70, 45){$S_V=\R X_0$}
& 
\includegraphics[width=3cm, angle=0]{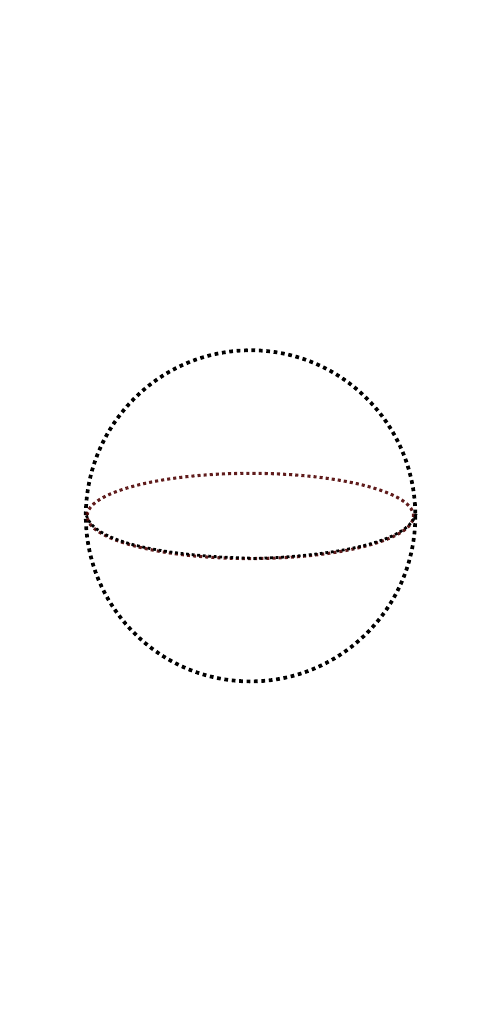}
\put(-50, 45){$S_V$}
\\ \\a) $\tau_0=\tau_{hy}$ && b) $\tau_0=\tau_{el}$ and  $\R E\ne
\emptyset$
& c)
$\tau_0=\tau_{el}$ and  $\R E=
\emptyset$
&d) $\tau_0=\tau_{em}$
\end{tabular}
\end{center}
\caption{Possibilities for $(X_0,\tau_0)$ and $S_V$}
\label{fig:SV}
\end{figure}

\medskip
Let $F$ be
a $\tau$-invariant class in $H_2(X_\sharp\setminus
L_\sharp;\Z/2\Z)$ having a  $\tau$-invariant 
representative $\F$, 
and define $\F_i=\F \cap X_i$.
Note that by construction we have $\partial \F_i\subset E$. 
Throughout the text, we  always assume that 
$\F$ satisfies the following conditions:
\begin{itemize}
\item either $\F\cap\R E=\emptyset$, or there exists a neighborhood
  $U$ of 
$\R E$ in $X_\sharp$ 
such that $\F\cap U\subset \R X_\sharp$
(i.e. $\F$ is either disjoint from $\R E$, or is  locally 
contained in $\R X_\sharp$ around $\R E$); 
\item one of the
two following assumptions hold:
\begin{enumerate}
\item[$(H_1)$]  $\F_0 \cup L_0$ is a cycle representing 
a multiple of $ V$ 
in
$H_2(X_0;\Z/2\Z)$;

\item[$(H_2)$]  $\tau_0=\tau_{el}$, and $\F_0\cup L_0=\D$ or 
$\F_0\cup L_0=\D\cup S_V$, where
$\D$  is   a
  $\tau$-invariant embedded disk with $\partial \D\subset E$ (all
  possibilities are depicted in
  Figure \ref{fig:H2}).
\begin{figure}[h!]
\begin{center}
\begin{tabular}{ccccc}
\includegraphics[width=3cm, angle=0]{Figures/H1_el2.pdf}
\put(-118, 140){$\D=L_0$}
\put(-118, 40){$\R X_0\setminus\D$}
\put(-10, 80){$S_V$}
& \hspace{5ex} &
\includegraphics[width=3cm, angle=0]{Figures/H1_el2.pdf}
\put(-118, 140){$\D\subset\F_0$}
\put(-118, 40){$\R X_0\setminus\D$}
\put(-10, 80){$S_V$}
& \hspace{5ex} &
\includegraphics[width=3cm, angle=0]{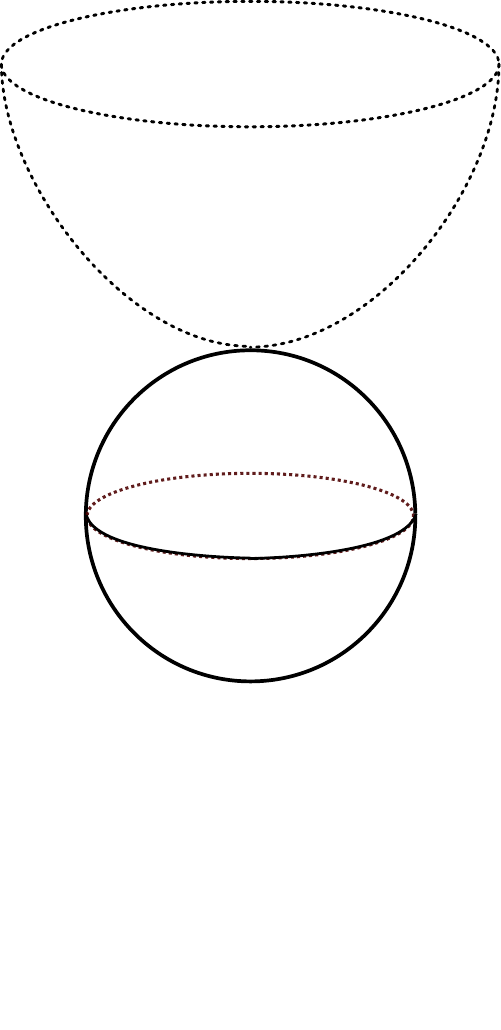}
\put(-118, 140){$\D\subset\F_0$}
\put(-70, 45){$S_V=\R X_0$}
\\ \\a)$\R E\ne\emptyset$ and $\F_0\subset S_V$&& b)$\R E\ne\emptyset$  and $L_0=\emptyset$
&& c) $\R E=\emptyset$  and  $L_0=\emptyset$
\end{tabular}
\end{center}
\caption{Possibilities for $\F_0\cup L_0$ under the assumption $(H_2)$}
\label{fig:H2}
\end{figure}
\end{enumerate}
\end{itemize}

\subsection{Vanishing Welschinger invariants}\label{sec:W vanish}
Next theorem is a key ingredient in the proof of Theorem
\ref{thm:chain spheres}, and 
will be proved in Section \ref{sec:proof vanish}.
\begin{thm}\label{thm:W vanish}
Suppose that $\F_0\cup L_0$ satisfies assumption $(H_1)$ and
contains $\R X_0$, and that $L_0$ is a disk.
\begin{enumerate}
\item If   $r\ge 2$,
then
 $$W_{X_{\R},L,F}(d,s)= 0.$$

\item If $r =1$ and $c_1(X)\cdot d-1\ge 2 $, then 
$$2^{\frac{c_1(X)\cdot d-4}{2}}\ | \ W_{X_\R,L,F}(d,s)
\quad\mbox{and}\quad (-1)^{\frac{ d^2  -c_1(X)\cdot d
    +2}{2}}W_{X_\R,L}(d,s)\ge 0.$$
\end{enumerate}
\end{thm}
Note that the assumptions of Theorem \ref{thm:W vanish} imply that
$\tau_0=\tau_{el}$ and $\R E\ne \emptyset$.
In the Lagrangian sphere contraction presentation, the condition that 
 $L_0$ is a disk translates to the condition
that $L\cap S_V$ is reduced to a single intersection point.

\subsection{From $X_1$ to $X$}\label{sec:relations}
Here we reduce the computation of Welschinger invariants of $X_\R$ to
enumeration of real $J$-holomorphic curves in $(X_{1},\omega_1,\tau_1)$ for
a $\tau_{1}$-compatible
 almost complex structure $J$ for which  $E$ is $J$-holomorphic.
 
\begin{defi}\label{def:mult}
Let $J$ be a  $\tau_1$-compatible almost complex structure on
$(X_1,\omega_1,\tau_1)$ for which the curve $E$ is $J$-holomorphic, and 
let $C_1$ be an immersed real rational $J$-holomorphic  curve 
intersecting  $E$ transversely. 
We denote  by $a$  the  number of
points in $\R C_1\cap \R E$,  by $b$ the number of pairs
of $\tau_1$-conjugated points in $C_1\cap E$, and by 
$m_{L_1,\F_1}(C_1)$ the number of intersection points of a half of
$C_1\setminus \R C_1$ with $L_1\cup \F_1$.
Finally, let 
$k\ge 0$ be an integer.
\begin{enumerate}
\item If $\F_0$ satisfies assumption $(H_1)$, then we define
$$\mu^0_{L_\sharp,\F_0,k}(C_1)=(-1)^{m_{L_1,\F_1}(C_1)+ \gamma(a+b)}
\sum_{k=a_k+2b_k}
{a \choose a_k}{b\choose b_k}$$
and
$$\mu^2_{L_\sharp,\F_0,k}(C_1)=\left\{ \begin{array}{ll}
         (-1)^{m_{L_1,\F_1}(C_1)+\gamma b}\ 2^b & \mbox{if $a=0$ and $k=b$};\\
         0 & \mbox{otherwise}.\end{array} \right.,$$
where $\gamma=0,1$ is such that $[\F_0\cup L_0]=\gamma V$ in $H_2(X_0 ;\Z/2\Z)$.

\item If $\F_0$ satisfies assumption $(H_2)$, then we define
$$\mu_{L_\sharp,\F_0,k}(C_1)=\left\{ \begin{array}{ll}
         (-1)^{m_{L_1,\F_1}(C_1)} & \mbox{if $k=a=b=0$};\\
         0 & \mbox{otherwise}\end{array} \right.$$
\end{enumerate}

\end{defi}

As above let $d\in
H_2(X;\Z)$
and  $r,s\in\Z_{\ge 0}$ such that  
$$c_1(X)\cdot d - 1= r+2s.$$
Choose a configuration $\x$ made of $r$ points in $L_1$ and  $s$ 
 pairs of $\tau$-conjugated 
points in $X_1\setminus \R X_1$.
Let $J$ be  a $\tau_1$-compatible almost complex structure 
for which  $E$ is $J$-holomorphic.

For each integer $k\ge 0$, we denote by $\mathcal C_{1,k}(d,\underline x,J)$
the set of all irreducible rational real $J$-holomorphic curves in
$(X_1,\omega_1,\tau_1)$ 
  passing through all points in $\x$, realizing the class
$d-k[E]$, and such that $L_1$ contains the 1-dimensional
part of $\R C_1$. 
For a generic choice of $J$ satisfying the above conditions, 
it follows from Lemma \ref{cor:transverse finite} and Proposition
\ref{prop:finite} that
the set
$\mathcal C_{1,k}(d,\underline x,J)$ is finite, 
and that any curve in $\mathcal C_{1,k}(d,\underline x,J)$ is nodal
and intersects 
$E$ transversely. 
Moreover $\mathcal C_{1,k}(d,\underline x,J)$ is non-empty only for finitely
many values of $k$.

We prove next theorem in Section \ref{sec:proof mainthm}.
Recall that notations have been introduced in Section \ref{sec:degeneration}.
\begin{thm}\label{mainthm}
Suppose that $L_1 \ne \emptyset$ if $r>0$.
Then for a generic choice of $J$,
 the two following claims hold.
\begin{enumerate}
\item If $\F_0$ satisfies assumption $(H_1)$, then,
with the convention that $\chi(\emptyset)=0$,
   one has 
$$W_{X_{\R},L,F}(d,s)=\sum_{k\ge 0}
\  \sum_{C_1\in\mathcal C_{1,k}(d,\underline x,J)}
\mu^{\chi(\R X_0)}_{L_\sharp,\F_0,k}(C_1).$$

\item If $\F_0$ satisfies assumption $(H_2)$,
then  one has
$$W_{X_{\R},L,F}(d,s)=
\  \sum_{C_1\in\mathcal C_{1,0}(d,\underline x,J)}
\mu_{L_\sharp,\F_0,0}(C_1).$$

\end{enumerate}
\end{thm}

Applying  Theorem \ref{mainthm}$(1)$
with $F=[\R X\setminus L]$,
 one obtains {\cite[Theorem 2.2]{BP}}.
Some instances of  Theorem
\ref{mainthm}$(1)$
when $\R X_0=S^1\times S^1$ 
have been known for sometimes, e.g. \cite{Br31,Br27,Kha,RasSal}. 
Since the publication of \cite{BP}, an
algebro-geometric
 proof  of   Theorem \ref{mainthm}$(1)$   appeared in
 \cite{Bru14} 
and in
\cite{IKS13} in 
 the particular cases when $X$ is a Del Pezzo surface of degree 
 two or more.
Theorem \ref{mainthm}$(2)$ immediately implies the following
corollary.
\begin{cor}\label{cor:W equal}
 Suppose that $V\in H_2(X\setminus L;\Z/2\Z)$ and that $\F_0$ satisfies assumption $(H_2)$. 
Then 
$$W_{X_\R,L,F}(d,s)= W_{X_\R,L,F+V}(d,s). $$
\end{cor}

\subsection{Applications of Theorem \ref{mainthm}$(1)$}
We do not explicitly use Theorem \ref{mainthm}$(1)$ in the proof of Theorem
\ref{thm:chain spheres}, nevertheless its proof is almost contained in the
proof of  Theorem \ref{mainthm}$(2)$. Theorem \ref{mainthm}$(1)$ has
many interesting applications, in particular in explicit computations
of Welschinger invariants, see \cite{Bru14,IKS13}.
We present two other consequences.

\medskip
We first
 relate some tropical
Welschinger invariants  to genuine Welschinger invariants of
the quadric ellipsoid.  We refer to \cite{IKS09} for the definition of
tropical Welschinger invariants.
 The only   homology classes  of
 $(X_0,\omega_{FS}\oplus\omega_{FS},\tau_{el})$ realized by real curves
 are  of the form $d(l_1+l_2)$ with $d\in\Z_{>0}$.
We say that a tropical curve in $\mathbb R^2$ is of class 
$(a,b)$
in
the tropical second Hirzebruch surface $\mathbb T\mathbb F_2$ if
its Newton polygon has vertices $(0,0)$, $(0,a)$, $(b,a)$, and
$(2a+b,0)$. 
We denote by $W_{\mathbb T\mathbb  F_2}(d)$ the irreducible tropical
Welschinger invariant of $\mathbb T\mathbb
F_2$ for curves of class $(d,0)$.
\begin{prop}\label{class-trop}
For any  $d\in\Z_{>0}$, we have 
$$W_{X_{0,el},S^2}(dl_1+dl_2,0)=W_{\mathbb T \mathbb F_2}(d). $$
\end{prop}
\begin{proof}
We consider the second Hirzebruch surface 
$\mathbb F_2$ equipped with its real structure induced by the blow up
at the origin of the real quadratic cone with equation
$x^2+y^2-z^2=0$.
We denote respectively by $h$ and $f$
the class in $H_2(\mathbb F_2;\Z)$ of a hyperplane section
and of a fiber.
According to \cite{Mik1},
if  $\x^\mathbb T$ is a tropically generic 
  configuration  of $4d-1$ points in
$\mathbb R^2$, then  any rational tropical  curve in
$\mathbb T\mathbb F_2$ of
 class
 $(d-k, 2k)$ and containing $\x^\mathbb T$ 
 has 
$4d$ unbounded edges of weight 1. Still by
\cite{Mik1}, this implies the existence of a generic configuration
$\x$ of 
$4d-1$ points in
$\mathbb R\mathbb F_2$ such that any real algebraic rational curve in
$\mathbb F_2$ of
 class
 $(d-k)h +2kf$ and containing $\x$ 
intersect the $(-2)$-curve only in real points.
Now the corollary follows from  Theorem \ref{mainthm}$(1)$
applied with $\R X_0=S^2$.
\end{proof}

It is proved in \cite{IKS04} that
given a real toric Del Pezzo surface $X$
equipped with its 
tautological real toric structure and
a class $d\in H_2(X;\Z)$, we have
$$W_{X_\R,\R X}(d,0)\ge W_{X_\R,\R X}(d,1).$$
The same idea we used in the proof of
Proposition \ref{class-trop} combined with
Theorem \ref{mainthm} and {\cite[Theorem 3.12]{Bru14}}
provide a natural generalization of this
formula in the particular cases when $X$ is  a Del Pezzo surface of
degree at least three.

\begin{prop}\label{decr}
Let $(X,\omega)$ be a  
symplectic
4-manifold symplectomorphic/deformation equivalent to  
a Del Pezzo surface of
degree at least three.
If  $X_{\R}=(X,\omega,\tau_1)$ and $X_{\R}'=(X,\omega,\tau_2)$ are
two real structures on $(X,\omega)$,  
then
for any $d\in H_2(X;\Z)$ one has
$$W_{X_{\R},L_1}(d,0)\ge  W_{X_{\R}',L_2}(d,0)\ge 0\quad \text{if}\quad
\chi(\R X)\le \chi(\R X').$$
\end{prop}
\begin{proof}
We first prove the proposition in the case when $(X,\omega)$ is
deformation equivalent to $\C
P^2$ blown up at  six points.
We consider $\C P^2$ and its blown up equipped with the standard
complex structure $J_{st}$.
Let us denote by $\C P^2_6(\kappa)$ 
the blow up of $\C P^2$ 
in $6-2\kappa$ real points and $\kappa$ pairs of conjugated points,
such that this 6 blown-up points do not lye on a conic, and
 no 3 of them lye on the same line.
We further denote by $\widetilde{\C P}^2_6(\kappa)$
the blow up of $\C P^2$ 
in $6-2\kappa$ real points and $\kappa$ pairs of conjugated points,
such that this 6 blown-up points  lye on a smooth real
conic with a non-empty real part, but
 no 3 of them lye on the same line.
 We denote by $E$
the strict transform of this conic in $\widetilde{\C P}^2_6(\kappa)$. 
We also denote by $\C P^2_6(4)$ the real structure on 
the blow up $\C P^2$ 
in $6$  points with a disconnected real part (see \cite{Sil89,Kol97}).
We have
$$\chi(\R P^2_6(\kappa))=-5+2\kappa \quad \mbox{for all
}\kappa\in\{0,1,2,3,4\}. $$
Note that $\C P^2_6(\kappa)$ contains no complex algebraic curves $C$
with $C^2\le -2$, and 
that the curve $E$ and its multiples are the only
algebraic curves in $\widetilde{\C P}^2_6(\kappa)$ with
self-intersection strictly less that $-1$. Hence it follows from
\cite[Lemma 3.3.1]{McS} that $J_{st}$ is 
generic enough for our purposes, as long as we consider generic configuration of
points in $\widetilde{\C P}^2_6(\kappa)$.
Theorem \ref{mainthm}$(1)$ applied to $E$ in 
$X_1=\widetilde{\C P}^2_6(\kappa)$ 
allows one to compute
$W_{\C P^2_6(\kappa),L_1}(d,s)$ (when $\R X_0=S^1\times S^1$) and 
$W_{\C P^2_6(\kappa+1),L_2}(d,s)$  (when $\R X_0=S^2$) out of the sets 
$\mathcal C_{1,k}(d,\x,J_{st})$.
When $s=0$, it follows from {\cite[Theorem 3.12]{Bru14}} that there exists a
configuration of real points 
$\x$ in $\widetilde{\C P}^2_6(\kappa)$ 
such that for any $k\ge 0$, any curve in $\mathcal C_{1,k}(d,\x,J_{st})$
intersects $E$ only in real points (i.e. $b=0$), and 
$$\sum_{C_1\in\mathcal C_{1,k}(d,\underline x,J_{st})}
(-1)^{m_{L_1,\R X_1\setminus L_1}(C_1)}\ge 0. $$
Hence by  Theorem \ref{mainthm}$(1)$ we obtain
$$W_{\C P^2_6(\kappa),L_1}(d,0) - W_{\C P^2_6(\kappa+1),L_2}(d,0)\ =\ 
 \sum_{k\ge 1} {d\cdot [E] \choose k} \sum_{C_1\in\mathcal C_{1,k}(d,\underline x,J_{st})}(-1)^{m_{L_1,\R X_1\setminus
     L_1}(C_1)}\ \ge 0.$$

The proof in the case of $\C P^1\times \C P^1$ is analogous using
floor diagrams from \cite{Br6b}.
\end{proof}

Note that Proposition \ref{decr} does not generalize immediately to any
symplectic 4-manifold. Indeed, according to {\cite[Section 7.3]{ABL}} one has
$W_{\C P^2,\R P^2}(9, 12)< W_{\C P^2,\R P^2}(9, 13)$, i.e. Proposition \ref{decr} does not
hold in the case of $\C P^2$ blown up in 26 points.

\section{Real symplectic sums and enumeration of real curves}\label{sec:proofs}

This section is devoted to the proof of Theorems \ref{thm:W vanish}
and \ref{mainthm}. We start by performing some preliminary
computations in Section \ref{sec:tangency}. We recall the symplectic
sum construction in Section
\ref{sec:sympl sum}, as well as a basic application to  
complex enumerative problems.
We  prove Theorems \ref{thm:W vanish}
and \ref{mainthm} in Sections \ref{sec:proof mainthm} and
\ref{sec:proof vanish}, by adapting results from Section
\ref{sec:sympl sum} to the real setting.

\medskip
An isomorphism between two $J$-holomorphic maps $f_1:C_1\to X$  and
$f_2:C_2\to X$ is a biholomorphism  $\phi:C_1\to C_2$ such that 
$f_1=f_2\circ\phi$.
Maps are always considered up to isomorphisms. 

\medskip
Given $\alpha=(\alpha_i)_{i\ge 1}\in \Z_{\ge
  0}^\infty$, we use the following notation:
$$|\alpha|=\sum_{i=1}^{+\infty} \alpha_i,\quad 
\mbox{and}\quad I\alpha=\sum_{i=1}^{+\infty} i\alpha_i. $$
The vector in
$\Z_{\ge 0}^\infty$ whose all coordinates are equal to $0$, except the
$i$th one which is equal to 1, is denoted by $e_i$.

\subsection{Curves with tangency conditions}\label{sec:tangency}
Let $(X,\omega)$ be a compact and connected $4$-dimensional symplectic
manifold, and let $E\subset X$ be an embedded symplectic curve in $X$.
Let $d\in H_2(X;\Z)$ and 
$\alpha,\beta\in \Z_{\ge 0}^\infty$
such that
$$I\alpha + I\beta=d\cdot [E].$$
Choose a configuration $\x=\x^\circ\sqcup \x_E$ of points in $X$,
with $\x^\circ$ a configuration of $c_1(X)\cdot d  -1  - d\cdot [E]+
|\beta|$ points 
in $X\setminus E$, and $\x_E=\{p_{i,j}\}_{0< j\le \alpha_i, i\ge 1}$
a configuration of 
$|\alpha|$ points in $E$.
Given $J$ an almost complex structure on $X$ tamed by $\omega$ and for
which $E$ is $J$-holomorphic, 
we denote by $\CC^{\alpha,\beta}(d,\x,J)$  the set
 of rational $J$-holomorphic maps $f:\C P^1\to X$ such that 
\begin{itemize}
\item $f_*[\C P^1]=d$;
\item $\x\subset f(\C P^1)$;
\item $E$ does not contain $f(\C P^1)$;
\item $f(\C P^1)$ has order of contact $i$ with $E$ at each points
  $p_{i,j}$;
\item $f(\C P^1)$ has order of contact $i$ with $E$ at exactly $\beta_i$
  distinct points on $E\setminus \x_E$.
\end{itemize}
For a generic choice of $J$, 
the set of simple maps in
$\CC^{\alpha,\beta}(d,\x,J)$ 
is 0-dimensional.
However 
 $\CC^{\alpha,\beta}(d,\x,J)$ 
might contain
components of positive dimension corresponding to non-simple maps.

\begin{lemma}\label{cor:transverse finite}
Suppose that  $\beta=(d\cdot [E])$ and $\alpha=0$, or 
$\beta=(d\cdot [E]-1)$ and $\alpha=(1)$. Then   for a
generic choice of $J$, the set
$\CC^{\alpha,\beta}(d,\x,J)$ only  contains simple maps.
\end{lemma}
\begin{proof}
 Suppose on the
contrary that $\CC^{\alpha,\beta}(d,\x,J)$ contains a
non-simple map 
 which factors through a non-trivial ramified covering of degree
$\delta$ of a simple
map $f_0:\C 
P^1\to X$. Let $d_0$ denotes the homology class 
$(f_0)_*[\C P^1]$. Since $f_0(\C P^1)$ passes through $\delta
c_1(X)\cdot d_0-1$ points, we have
$$c_1(X)\cdot d_0-1\ge \delta c_1(X)\cdot d_0-1\ge 0,$$
which is impossible.
 \end{proof}

Next proposition shows
that
the set of images of non-simple maps in $\CC^{\alpha,\beta}(d,\x,J)$
is 0-dimensional.

\begin{prop}\label{prop:finite covering}
Suppose that  $\CC^{\alpha,\beta}(d,\x,J)$ contains a
non-simple map 
$f$ which factors through a non-trivial ramified covering of a simple
map $f_0:\C 
P^1\to X$. Denote by $d_0$ the homology class 
$(f_0)_*[\C P^1]$,
 and let $\alpha',\beta'\in \Z_{\ge 0}^\infty$ such that
$f_0 \in \CC^{\alpha',\beta'}(d_0,\x,J)$. 
Then for a generic choice of $J$, we have
$$c_1(X)\cdot d_0  -1  - d_0\cdot [E]+ |\beta'|=|\x^\circ|=k_1\quad \mbox{and}\quad
|\alpha'|=k_2$$
 with 
$(k_1,k_2)=(1,0), (0,1),$ or $(0,0)$.
Moreover in the first two cases,
 the set 
of such ramified coverings $f$ 
is finite,
 and
$|\alpha'|+|\beta'|\ge 2$.
\end{prop}
\begin{proof}
Let $\delta\ge 2$ be the degree of the covering map through which $f$
factors. In particular we have $d=\delta d_0$.
By Riemann-Hurwitz Formula, we have 
$$\delta (|\alpha'| + |\beta'|)-|\alpha|-|\beta|\le 2\delta-2.$$
Combining the latter identity with $|\alpha'|=|\alpha|$, we get
\begin{equation}\label{eq:RH}
|\beta|-|\beta'|\ge (\delta-1)(|\alpha'|+ |\beta'|-2).
\end{equation}
Since $f_0(\C P^1)$ contains all points in $\x^\circ$, we have
$$c_1(X)\cdot d_0  -1  - d_0\cdot [E]+
|\beta'|\ge |\x^\circ|= \delta c_1(X)\cdot d_0  -1  - \delta d_0\cdot [E]+
|\beta|,$$
and so
$$(\delta-1)(d_0\cdot [E] - c_1(X)\cdot d_0)\ge  |\beta|-|\beta'|.$$
Combining this identity with  $(\ref{eq:RH})$, we obtain
$$0\ge (\delta-1)( c_1(X)\cdot d_0- d_0\cdot [E] + |\alpha'|+ |\beta'|-2). $$
Since we have
$$\delta\ge 2, \quad c_1(X)\cdot d_0- d_0\cdot [E] +  |\beta'|-1\ge
0, \quad \mbox{and}\quad |\alpha'|\ge 0, $$
 we deduce that 
$$c_1(X)\cdot d_0  -1  - d_0\cdot [E]+ |\beta'|=k_1\quad \mbox{and}\quad
|\alpha'|=k_2$$ 
with 
$(k_1,k_2)=(1,0), (0,1),$ or $(0,0)$.
Moreover in the first two cases, all inequalities above are in fact
equalities. In particular there exists finitely many coverings 
$\pi:\C P^1\to \C P^1$ 
of
degree $\delta$ such that $f_0\circ \pi\in
\CC^{\alpha,\beta}(d,\x,J)$.
 Since $|\beta|-|\beta'|\ge 0$, we also
deduce from $(\ref{eq:RH})$ that $|\alpha'|+|\beta'|\ge 2$.
\end{proof}

\begin{rem}
The three cases from Proposition \ref{prop:finite covering} show up,
even in simple situations. Let us consider for example $X$ to be $\C
P^2$ blown up at a point $q$.  Denote by $l$ the homology class of a
line, by $l_{exc}$ the class of the
exceptional divisor,  and by $E$ the pull back of
 a conic not passing through $q$.
Then for any choice of $J$, the sets
$\CC^{0,2e_\delta}(\delta(l-l_{exc}),\{p\},J)$,
$\CC^{e_\delta,e_\delta}(\delta(l-l_{exc}),\emptyset,J)$, and
$\CC^{0,e_{2\delta}}(\delta(l-l_{exc}),\emptyset,J)$ with $\delta\ge 2$ contain a
non-trivial ramified covering of a line,  the third set being of
dimension $\delta-1$. 
\end{rem}

\begin{prop}\label{prop:finite}
Suppose that $E$ is an embedded symplectic sphere with $[E]^2\ge -2$,
and that $|\beta| \ge [E]\cdot d -1$.
Then for a generic choice of $J$, the set $\CC^{\alpha,\beta}(d,\x,J)$ contains finitely many
simple maps. As a consequence, the set
$$\CC_{*}^{\alpha,\beta}(d,\x,J)=\left\{f(\C P^1) \ | \ (f:\C P^1\to X)\in
\CC^{\alpha,\beta}(d,\x,J)\right\}$$  
is also finite.
\end{prop}
\begin{proof}
Suppose that $\CC^{\alpha,\beta}(d,\x,J)$ contains infinitely many
simple maps. By Gromov compactness Theorem, there exists a sequence
$(f_n)_{n\ge 0}$ of simple maps in $\CC^{\alpha,\beta}(d,\x,J)$ which
converges to some $J$-holomorphic map $\overline f:\overline C \to
X$. By  genericity of $J$, 
the set of simple maps in  $\CC^{\alpha,\beta}(d,\x,J)$ is discrete. Hence
 either $\overline C$ is reducible, or $\overline f$ is non-simple.
Let 
$\overline C_1,\ldots , \overline C_m,\overline C'_1,\ldots,\overline C'_{m'}$ 
be the irreducible components of $\overline C$, labeled in such a way that
\begin{itemize}
\item  $\overline  f(\overline C_i)\not\subset E$;
\item $\overline f(\overline C'_i)\subset E$, and $\overline
  f_*[\overline C'_i]= k_i[E]$.
\end{itemize}
Define $k=\sum_{i=1}^{m'}k_i$.
The 
restriction of $\overline f$ to
$\bigcup_{i=1}^m\overline C_i$ is subject to $c_1(X)\cdot d -1 -d\cdot
[E] +|\beta|$ points conditions, so we have
$$c_1(X)\cdot (d-k[E]) -m\ge c_1(X)\cdot d-1  -d\cdot [E] +|\beta|.$$
Since $E$ is an embedded sphere, the adjunction formula implies  that 
$c_1(X)\cdot [E]=[E]^2+2$. Hence we get
$$c_1(X)\cdot d- 2k -k[E]^2 -m\ge c_1(X)\cdot d-1  -d\cdot [E]
+|\beta|,$$
that is
$$0\ge -d\cdot [E]+|\beta|+m-1+2k +k[E]^2. $$
Since  $d\cdot [E]\ge|\beta|$, we are in one of the following situations:
\begin{enumerate}
\item $d\cdot [E]=|\beta|$ (in particular $\alpha=0$):
\begin{enumerate}
\item  $k=0$, and $m=1$;
\item $[E]^2=-2$, $k>0$, and $m=1$;
\end{enumerate}
\medskip

\item $d\cdot [E]=|\beta|+1$ (in particular either 
$\beta=(d\cdot [E]-1)$ and $\alpha=(1)$, or $\beta=(d\cdot [E]-2,1)$):
\begin{enumerate}
\item $k=0$, and $m=1$;
\item $k=0$, and $m=2$;
\item $[E]^2=1$, $k=1$, and $m=1$;
\item $[E]^2=-2$, $k>0$, and $m=1$;
\item $[E]^2=-2$, $k>0$, and $m=2$.
\end{enumerate}
\end{enumerate}
We end the proof of the proposition by ruling out all these cases one by one.
\begin{itemize}
\item[(1)(a)] $d\cdot [E]=|\beta|$, $k=0$, and $m=1$:

As explained above, the map $\overline f$ has to factorize through 
 a non-trivial ramified covering of a simple
map $f_0:\C 
P^1\to X$. But then $f_0$ is subject to more point
constraints that the dimension of its space of deformation, which provides
a contradiction. 

\item[(1)(b)] $d\cdot [E]=|\beta|$, $[E]^2=-2$, $k>0$, and $m=1$:

By genericity, 
the curve $\overline f(\overline C_1)$ is fixed by the $c_1(X)\cdot
d-1$ point constraints and intersect $E$ transversely. 
Any intersection point of $\overline f(\overline
C_1\setminus (\overline C'_1\cup\ldots\cup \overline C'_{m'}))$ and
$E$ 
 deforms to
an intersection point of $f_n(\C P^1)$ and $E$ 
for $n>>1$.
Since
$(d-k[E])\cdot [E]=d\cdot [E] +2k$, 
at least $d\cdot [E] +k$ intersection points of $\overline f(\overline
C_1)$ and $E$ deform to an intersection point of $f_n(\C P^1)$ and $E$
for $n>>1$. But this contradicts the fact that two $J$-holomorphic
curves intersect positively.

\item[(2)(a)] $d\cdot [E]=|\beta|+1$, $k=0$, and $m=1$:

Since $\overline f$ is a non-simple map, it factorizes through 
 a non-trivial ramified covering of degree $\delta\ge 2$ of a simple
map $f_0:\C 
P^1\to X$. If $d_0=(f_{0})_*[\C P^1]$,
the adjunction formula implies that
 the image of $f_0$ has 
$$\frac{d_0^2 -c_1(X)\cdot d_0 +2}{2} $$
nodes. Each of this node deforms to $2\delta$
 intersection point of $f_n(\C P^1)$ and  $ f_0(\C P^1)$
for $n>>1$. Since $\x\subset f_n(\C P^1)\cap f_0(\C P^1)$, we get
$$d\cdot d_0 \ge 2\delta\ \frac{d_0^2 -c_1(X)\cdot d_0 +2}{2} 
+ c_1(X)\cdot d -2 = d\cdot d_0 +2(\delta-1)> d\cdot d_0$$
which is a contradiction.

\item[(2)(b)] $d\cdot [E]=|\beta|+1$,  $k=0$, and $m=2$:

By genericity, 
the curve $\overline f(\overline C_1\cup\overline C_2)$
 is fixed by the $c_1(X)\cdot
d-2$ point constraints, and intersect $E$ transversely at
non-prescribed points. This contradicts the fact that either
$\alpha\ne 0$ or 
$\beta_2\ne 0$.

\item[(2)(c)] $d\cdot [E]=|\beta|+1$,  $[E]^2=1$, $k=1$, and $m=1$:

By genericity, 
the curve $\overline f(\overline C_1)$ is fixed by the $c_1(X)\cdot
d-2$ point constraints, and intersect $E$ transversely in
$d\cdot [E]+1$ non-prescribed  points. 
Any such intersection point distinct from   
$\overline f(\overline C_1\cap\overline C'_1)$ deforms to an
intersection
point of $f_n(\C P^1)$ and $E$ for $n>>1$.
Moreover since all intersection points of $\overline f(\overline C_1)$ and
$E$ are transverse and non-prescribed, 
the component $\overline C'_1$ contains the limit of
the point corresponding to the extra constraint $\alpha_1$ or $\beta_2$.
Therefore $f_n(\C P^1)$ and $E$ for $n>>1$ must have at least $d\cdot
[E]+1$ intersection points for $n>>1$,
 which is a contradiction.

\item[(2)(d)] $d\cdot [E]=|\beta|+1$,  $[E]^2=-2$, $k>0$, and $m=1$:

Suppose first that $\overline f_{|\overline C_1}$  factorizes through  
 a non-trivial ramified covering  of degree $\delta\ge 2$ of a simple
map $f_0:\C 
P^1\to X$. Since $f_0(\C P^1)$ satisfies
$c_1(X)\cdot d -2$ point conditions, we have that
$$c_1(X)\cdot d_0 -1 \ge c_1(X)\cdot d -2 = \delta c_1(X)\cdot d_0
-2\ge 0,  $$
where $d_0=(f_{0})_*[\C P^1]$.
Hence we obtain  that $c_1(X)\cdot d_0=1$ and $\delta=2$. In
particular the curve $f_0(\C P^1)$ is rigid, and intersect $E$ at smooth
points. Now the same arguments used in the case (2)(a) provide a
contradiction.

Hence $\overline f_{|\overline C_1}$ is a simple map. 
Since $\overline f_{|\overline C_1}$ satisfies $c_1(X)\cdot d-2$ point constraints,
it has at most one tangency point with $E$. The same argument used in
the case (1)(b) implies that $k=1$ and $\overline f(\overline C_1)$ is
tangent to $E$ at $\overline f(\overline C_1\cap \overline
C'_1)$. Hence  $\overline f_{|\overline C_1}$ is fixed by this
tangency condition and the $c_1(X)\cdot d -2$ other point
conditions, and the component $\overline C'_1$ contains the limit of
the point corresponding to the extra constraint $\alpha_1$ or $\beta_2$. Thus
we obtain again a contradiction with the positivity of intersection
points of $E$ and $f_n(\C P^1)$ for $n>>1$.

\item[(2)(e)] $d\cdot [E]=|\beta|+1$,  $[E]^2=-2$, $k>0$, and $m=2$:

By genericity, 
the curve $\overline f(\overline C_1\cup\overline C_2)$
 is fixed by the $c_1(X)\cdot
d-2$ point constraints, and intersect $E$ transversely at
non-prescribed points. Hence the same argument used in
the case (1)(b) implies that $k=1$. Thus the component $\overline C'_1$ contains the limit of
the point corresponding to the extra constraint $\alpha_1$ or
$\beta_2$, which gives a contradiction as in the case (2)(d).
\end{itemize}

The finiteness of the set
$\CC_{*}^{\alpha,\beta}(d,\x,J)$ follows from Proposition
\ref{prop:finite covering} and the finiteness of simple maps in
$\CC^{\alpha',\beta'}(d_0,\x,J)$ for all possible
$\alpha',\beta',$ and $d_0$ with $d=\delta d_0$. 
\end{proof}

 In the case when $X=\C P^1\times \C P^1$, $[E]=l_1+l_2$, and
 $|\x^\circ|\le 1$,  
the set $\CC^{\alpha,\beta}(d,\x,J)$ is always
finite and made of simple maps. 

\begin{prop}\label{prop:quadric}
Suppose that
 $X=\C P^1\times \C P^1$ and $[E]=l_1+l_2$. Then  the set
$\CC^{\alpha,\beta}(d,\x,J)$ with $|\x^\circ|\le 1$ is empty for a
generic choice of $J$,
except in the following situations where it contains a unique
element:
\begin{itemize}
\item $\CC^{e_1,0}(l_i,\emptyset,J)$, $i=1,2$;
\item $\CC^{0,e_1}(l_i,\{p\},J)$, $i=1,2$;
\item $\CC^{2e_1,0}(l_1+l_2,\{p\},J)$;
\item $\CC^{e_2,0}(l_1+l_2,\{p\},J)$.
\end{itemize}
Moreover, this unique element is an embedding.
\end{prop}
\begin{proof}
The first Chern class of $X$ is dual to $2(l_1+l_2)$, so 
$$c_1(X)\cdot (al_1+bl_2)  -1  -  (al_1+bl_2)\cdot [E]+
|\beta| = a+b -1 +|\beta|.$$

Suppose that  $a+b -1 +|\beta|=0$ and that
$\CC^{\alpha,\beta}(al_1+bl_2,\emptyset,J)\ne \emptyset$. Since 
$(al_1+bl_2)\cdot [E]=a+b\ge |\beta|$, and since two
$J$-holomorphic curves 
intersect positively, we obtain  $a+b=1$ and $|\beta|=0$.
By genericity of $J$, we have that $(al_1+bl_2)^2\ge -1$, i.e. $2ab\ge -1$.
From $a+b=1$, we deduce that $a=0$ or $1$.

\medskip
In the case  $a+b -1 +|\beta|=1$ and 
$\CC^{\alpha,\beta}(al_1+bl_2,\{p\},J)\ne \emptyset$,
  we prove analogously that we are in one of the following situations:
\begin{itemize}
\item $(a,b)=(1,0)$ or $(0,1)$, and $|\beta|=1$;
\item  $(a,b)=(1,1),(2,0)$, or $(0,2)$, and $|\beta|=0$.
\end{itemize}

\medskip
If $X$ is equipped  with the  symplectic form $\omega_{FS}\oplus\omega_{FS}$
and  its standard complex structure $J_{st}$, 
it is easy to check that the sets
$\CC^{e_1,0}(l_i,\emptyset,J_{st})$, 
$\CC^{0,e_1}(l_i,\{p\},J_{st})$,
 $\CC^{2e_1,0}(l_1+l_2,\{p\},J_{st})$, and
 $\CC^{e_2,0}(l_1+l_2,\{p\},J_{st})$ consists of a unique element.
This implies that when we vary both $\omega$ and $J$, the
corresponding sets
still contain at least one element. Moreover they cannot contain more
than one element, since the imposed constraints imply that
 two distinct curves would
have an intersection number strictly bigger than the one imposed by
their homology class.
Finally,  all  $J$-holomorphic maps under consideration are
embeddings thanks to the adjunction formula.

\medskip
Suppose now that $\CC^{\alpha,0}(2l_i,\{p\},J)$ contains an element
$f:\C P^1\to X$. We  proved in the previous paragraph 
that there exists a map $f_0:\C P^1\to X$ in
$\CC^{0,e_1}(l_i,\{p\},J)$. Since we have $f_*[\C P^1]\cdot (f_0)_*[\C
  P^1]=2l_i^2=0$, we deduce that $f$ factors through $f_0$ and a degree
2 ramified 
covering of $\C P^1$. This contradicts Proposition
\ref{prop:finite covering}.
\end{proof}

\subsection{Symplectic sums}\label{sec:sympl sum}
Here we describe a
very particular case of
 the symplectic sum 
formula from
\cite{IP}.
Recall that $(X_1,\omega_1)$ is a compact and connected symplectic
manifold of dimension 4, containing an embedded symplectic sphere $E$
with $[E]^2=-2$. We furthermore assume the existence of a
symplectomorphism $\phi$ from $E$ to 
 a symplectic curve
realizing the class $l_1+l_2$ in 
$(X_0,\omega_0)=(\C P^1\times \C P^1,\omega_{FS}\oplus\omega_{FS})$. 
 By abuse, we still denote by $E$ the image $\phi(E)$ in $X_0$.
Since the self-intersection of $E$ in $X_0$ and $X_1$ are opposite, 
there exists a symplectic bundle isomorphism $\psi$
between the  normal bundle of $E$ in $X_0$ and  
the dual of the normal bundle of $E$ in $X_1$.
Out of these data, one produces a family of symplectic
$4$-manifolds $(Y_t,\omega_{t})$ parametrized by a small complex
number $t$ in $\C^*$, see \cite{Gom95}. All those manifolds are deformation
equivalent, and are called symplectic sums of $(X_0,\omega_0)$ and
$(X_1,\omega_1)$ along $E$.
Next theorem says that this family can be seen as a symplectic
deformation of the singular 
symplectic manifold $X_\sharp=X_0\cup_E X_1$ 
obtained by gluing $(X_0,\omega_0)$ and  $(X_1,\omega_1)$ along
$E$.

\begin{prop}[{\cite[Theorem 2.1]{IP}}]\label{thm:degen}
There exists a symplectic $6$-manifold $(Y, \omega_Y)$
and a symplectic fibration $\pi : Y\to D$ over a disk $D\subset\C$ such that
the central fiber $\pi^{-1}(0)$ is the singular symplectic
manifold $X_\sharp$, and $\pi^{-1}(t)=(Y_{t},\omega_{t})$
for $t\ne 0$.
\end{prop}
Topologically, $Y_t$ is obtained by removing a tubular
neighborhood of $E$ in $X_1$, and gluing back $X_0\setminus E$ via
$\psi$. Note that the symplectomorphism  $\phi$ induces a diffeomorphism $\Psi$ from
the normal bundle of $E$ in $X_1$ and $X_0\setminus E$.
Hence the homology groups $H_2(X_1;\Z)$ and 
 $H_2(Y_t;\Z)$ are identified, the class $[E]$ being identified with
the class $\Psi_*[E]$. 
Without loss of generality we may assume 
that $\Psi_*[E]=l_1-l_2$ in $H_2(X_0;\Z)$.

Let $d \in H_2(Y_{t};\Z)$, and 
choose  $\x(t)$ 
a  set of $c_1(X)\cdot d-1$ 
symplectic sections $ D \to  Y$ such that $\x(0)\cap
E=\emptyset$. Choose an almost complex structure
$J$ on $Y$ tamed by $\omega_Y$, which restrict to 
an almost complex structure
$J_t$ tamed by $\omega_t$ on each fiber 
$Y_t$,
and generic with respect to all choices
we made. 

Define $\CC(d,\x(0),J_0)$ to be
the set $\left\{\overline f:\overline C  \to X_\sharp \right\}$
of limits, as stable maps, 
 of
  maps in $\CC(d,\x(t),J_t)$ as $t$
  goes to $0$.
Recall (see {\cite[Section 3]{IP}}) that $\overline C$ is a connected nodal
rational curve  such that:
\begin{itemize}
\item $\x(0)\subset \overline f(\overline C)$;
\item any point $p\in \overline f^{\ -1}(E)$ is a node of 
$\overline C$ which is the intersection of two 
irreducible components $\overline C'$ and $\overline C''$ of
$\overline C$, with  $\overline f(\overline
C')\subset X_0$ and $\overline f(\overline
C'')\subset X_1$;

\item if in addition neither $\overline f(\overline
C')$ nor $\overline f(\overline
C'')$ is entirely mapped to $E$, then
the multiplicity of intersection
of both 
$\overline f(\overline
C')$ and $\overline f(\overline
C'')$ with $E$ are equal.
\end{itemize}

Given  an element $\overline f:\overline C\to X_\sharp$ 
of $\CC(d,\x(0),J_0)$, we denote by $C_i$  the union of the irreducible
components of $\overline C$ mapped to $X_i$.
\begin{lemma}\label{lem:homology}
Given an element $\overline f:\overline C\to X_\sharp$
of $\CC(d,\x(0),J_0)$,
there exists
$k\in\Z_{\ge 0}$ such that
$$
\overline f_*[C_1] = d -  k[E] \quad \textrm{and} \quad
\overline f_*[C_0]= kl_1 + (d\cdot [E] +k)l_2 .
$$
Moreover $c_1(X_1)\cdot \overline f_*[C_1]=c_1(Y_t)\cdot d$.
\end{lemma}
\begin{proof}
Let $k$ such that $f_*[C_1] = d -  k[E]$, and let 
$f_*[C_0]= al_1 +bl_2$.
Using the above identification of $H_2(X_1;\Z)$ and 
 $H_2(Y_t;\Z)$,
we have 
$$d\cdot [E]= \overline  f_*[C_0]\cdot (l_1-l_2)=b-a.$$ 
On the other hand, by considering a representative of $E$ in $X_1$ and
another in $X_0$, we obtain
$$a+b=\overline f_*[C_0]\cdot (l_1+l_2)= (d -  k[E])\cdot [E]= d\cdot [E]+2k, $$
which gives $a=k$ and $b=d\cdot [E] +k$.

By \cite[Lemma 2.2]{IP}, we have
$$c_1(Y_t)\cdot d = c_1(X_1)\cdot \overline f_*[C_1] +  c_1(X_0)\cdot
\overline f_*[C_0] -  2\overline f_*[C_0]\cdot [E]. $$
Since $c_1(X_0)$ is dual to $2(l_1+l_2)$, we deduce that $c_1(X_1)\cdot \overline f_*[C_1]=c_1(Y_t)\cdot d$.
\end{proof}

\begin{prop}\label{prop:degeneration}
Assume that the set $\x(0)\cap X_0$ contains at most one point, and
that $\x(0)\cap X_1\ne\emptyset$ if $\x(0)\cap X_0\ne\emptyset$.
Then for a generic   $J_0$,  the set
$\CC(d,\x(0),J_0)$ is finite, and only depends on $\x(0)$ and
$J_{0}$.
 Given
$\overline f:\overline C\to  X_\sharp$
  an element of  $\CC(d,\x(0),J_0)$, the
restriction of $\overline f$ to any component of $\overline C$ is a
simple map, and
no irreducible component of $\overline C$ is
  entirely mapped to $E$. Moreover the following are true.
\begin{enumerate}
\item If $\x(0)\cap X_0=\emptyset$, then 
 the curve $C_1$ is irreducible, 
and the image of any irreducible component of $C_0$ realizes
a class $l_i$.
The map $\overline f$ is the limit of a unique element of
$\CC(d,\x(t),J_t)$ as $t$ goes to 0.

\medskip
\item If $\x(0)\cap X_0=\{p_0\}$,
then
the image of the irreducible component $\overline C'$ of $C_0$
 whose
image contains $p_0$ realizes
either a class $l_i$ or the class $l_1+l_2$, while any other irreducible component of $C_0$ realizes a class $l_i$.

\begin{enumerate}
\item If $\overline f(\overline C')$ realizes the class  $l_i$, then the curve $C_1$ is irreducible and
  $\overline f_{|C_1}$ is an element of
$\CC^{e_1, (d\cdot [E]+2k-1)e_1}(d -  k[E],\x(0)\cup \x_E, J_0)$, where
  $\x_E=\overline f(\overline C')\cap E$. The map $\overline f$ is the limit of a unique element of
$\CC(d,\x(t),J_t)$ as $t$ goes to 0.

\item If $\overline f(\overline C')$ realizes the class   $l_1+l_2$, then  $\overline f_{|\overline C'}$ 
is an element of
$\CC^{\alpha, 0}(l_1+l_2,\{p_0\}\cup \x_E, J_0)$,
where
  $\x_E\subset \overline f( C_1)\cap E$, and $\alpha=2e_1$ or
$\alpha=e_2$. 
In the former case, the curve $C_1$ has two irreducible components, 
and $\overline f$ is the limit of a unique element of
$\CC(d,\x(t),J_t)$ as $t$ goes to 0; in the latter case,
the curve $C_1$ is irreducible, and $\overline f$ is the limit of
exactly two elements of
$\CC(d,\x(t),J_t)$ as $t$ goes to 0.
\end{enumerate}

\end{enumerate}
\end{prop}
\begin{proof}
The fact that no component of $\overline C$ is entirely mapped to $E$
 follows from
{\cite[Example 11.4 and Lemma 14.6]{IP}}. 
By assumption we have $[E]^2=-2$ in $X_1$, so the adjunction
formula implies that $c_1(X_1)\cdot [E]=0$.  Since the curve $\overline f( C_1)$ passes through all the points in
$\x(0)\cap X_1$ and realizes the class $d -  k[E]$ in
$H_2(X_1;\Z)$, the following hold. 

\begin{enumerate}
\item If $\x(0)\cap X_0=\emptyset$, then 
the map $\overline f_{|C_1}$ is constrained by 
$c_1(X)\cdot d-1 = c_1(X_1)\cdot (d -  k[E]) -1$ points in $X_1$. 
Hence the curve $C_1$ is irreducible, the map $\overline
f_{|C_1}$ is simple, and  $\overline f( C_1)$
intersects $E$ transversely in $d\cdot [E]+2k$ distinct points. 
The curve $\overline C$ is rational, and any $J_0$-holomorphic curve in
$X_0$ intersects $E$, so we deduce that the curve $C_0$ has
exactly $d\cdot [E]+2k$ irreducible components. Furthermore 
the image of any of them realizes
a class $l_i$.

\medskip
\item If $\x(0)\cap X_0=\{p_0\}$,
then
the map $\overline f_{| C_1}$ is constrained by 
$c_1(X_1)\cdot (d -  k[E]) -2$ points in $X_1$.
Hence we are in one of the following
situations.

\begin{enumerate}
\item The curve $C_1$ is irreducible, and
$\overline f^{\ -1}(E)$ consists in $d\cdot [E]+2k$ distinct points. 
As above, the curve $C_0$ must have
exactly $d\cdot [E]+2k$ irreducible components, and 
the image of any of them realizes
a class $l_i$. Since $\overline f(\overline C')$  contains
$p_0$, the map
$\overline f_{|C_1}$ is also constrained by the point $\overline f(\overline C')\cap E$.
Hence  $\overline
f_{| C_1}$ is a simple map by Lemma \ref{cor:transverse finite}.

\item The curve $ C_1$ has two connected components, the map $\overline
f_{|C_1}$ is simple and fixed by $\x(0)\cap X_1$, and  $\overline f( C_1)$
intersects $E$ transversely in $d\cdot [E]+2k$ distinct points. Hence
the curve $ C_0$ must have
exactly $d\cdot [E]+2k-1$ irreducible components, one of them, say
$\overline C''$,
intersecting the two components of $C_1$. The curve
$\overline f(\overline C'')$ has
to realize the class $l_1+l_2$, and  the image of any other
irreducible component of $ C_0$ realizes a class $l_i$. Since
all these latter components are constrained by 
$\overline f( C_1)\cap E$, we deduce that $\overline
C'=\overline C''$.

\item The curve $C_1$ is irreducible,  and 
$\overline f^{\ -1}(E)$ consists in $d\cdot [E]+2k-1$ distinct points. 
Again, the curve $ C_0$ must have
exactly $d\cdot [E]+2k-1$ irreducible components, the image of one of
them being tangent to $E$. 
As in the case (2)(b), we deduce that this component must be
$\overline C'$, that its image must realize the class $l_1+l_2$, and that 
 the image of any other
irreducible component of $C_0$ realizes a class $l_i$. 

Suppose that $\overline f$ restricts to a non-simple map on $C_1$,
and
let $\rho:\C P^1\to \C P^1$ be the ramified covering through which 
$\overline f_{| C_1}$
factors.
 Since
 $\x(0)\cap X_1\ne\emptyset$, Proposition \ref{prop:finite covering}
implies that  at least two
ramification points of $\rho$ should be mapped to $E$.
 Hence there should exist
an irreducible component $\overline C''$  of $C_0$ 
distinct from $\overline C'$
and intersecting $E$ non-transversely. This contradicts the fact that
$\overline f_*[\overline C'']= l_i$.
\end{enumerate}

\end{enumerate}
The statement about the number of elements of
$\CC(d,\x(t),J_t)$  converging to $\overline f$ as $t$ goes to 0
follows from  {\cite{IP}}. Let us recall briefly the behavior, close to
a smoothing 
of an intersection point $p$ of $ C_1$  and $ C_0$, of  an 
elements $f_t;C_t\to Y_t$ of
$\CC(d,\x(t),J_t)$  converging to $\overline f$.
In local coordinates $(t,x,y)$ at $\overline f(p)$, the manifold
$Y$ is given by the equation $xy=t$, the manifold $X_0$
(resp. $X_1$) being
locally given by $\{t=0\ \mbox{and}\ y=0\}$ 
(resp. $\{t=0\ \mbox{and}\ x=0\}$). If the order of
intersection of $\overline f_{ C_0}$ and $E$ at $f(p)$ 
is equal to $s$, then the maps $\overline f_{ C_0}$ and
$\overline f_{ C_1}$ have expansions
$$ x(z)=az^s + o(z^s)\quad\mbox{and}\quad y(w)=bw^s +o(w^s), $$
where $z$ and $w$ are local coordinates at $p$ of $ C_0$ and
$ C_1$ respectively. 
For $0<|t|<<1$, there exists a solution
 $\mu(t)\in\C^*$ of  
$$\mu(t)^s=\frac{t}{ab},$$
such that the smoothing of $\overline C$ at $p$ is locally given by
$zw=\mu(t)$, and the map $f_t$ is approximated by the map
$$\{zw=\mu(t)\}\subset \C^2\mapsto (t,az^s,bw^s) $$
close to the smoothing of $p$ (see {\cite[Section 6]{IP}}, and also
  {\cite[Section 6.2]{TehZin14}} for details).
Furthermore, such maps $f_t \in \CC(d,\x(t),J_t)$ converging to
$\overline f$ are in one to one correspondence with a choice of such
$\mu(t)$ for each point of $C_0\cap C_1$. 
\end{proof}

Next Corollary generalizes Abramovich-Bertram-Vakil formula.
\begin{cor}\label{lem:lem2}
Suppose that $\x(0)\cap X_0=\emptyset$, and let $\overline f:\overline C\to  X_\sharp$
be  an element of  $\CC(d,\x(0),J_0)$.
Define 
 $\CC_{\overline f}$ to be the set of elements  $\overline
 f':\overline C'\to X_\sharp$ in $\CC(d,\x(0),J_0)$ such that 
 $\overline f_{| C_1}=\overline f'_{| C'_1}$.
If $\overline f_*[ C_1]=d-k[E]$, then $\CC_{\overline f}$ has
exactly  $\binom{d\cdot 
  [E]+2k}{k}$ elements.
\end{cor}
\begin{proof}
It follows  from Proposition \ref{prop:degeneration} and 
Lemma \ref{lem:homology}
that
 $\overline f_*[ C_1]=d-k[E]$ if and only if 
the image of exactly $k$ irreducible components of $C_0$ 
realize the class $l_1$.
Since $(d-k[E])\cdot [E]= d\cdot[E]+2k$,
the result follows.
\end{proof}

\subsection{Proof of Theorem \ref{mainthm}}\label{sec:proof mainthm}
Theorems \ref{thm:W vanish} and \ref{mainthm} are obtained by
considering a real version of the symplectic sum described in Section
\ref{sec:sympl sum}. We first  provide the proof of Theorem
\ref{mainthm}
 since it is a immediate adaptation of Proposition
 \ref{prop:degeneration} to the real setting. 
We equip the disc $D$ from Proposition
 \ref{thm:degen} with the standard complex conjugation, the 
 symplectic manifold $(X_i,\omega_i)$ with a real structure $\tau_i$,
 and 
$(Y,\omega_Y)$  with a
 real structure $ \tau_Y$ such that the
 map $\pi:Y\to D$ is real.
Furthermore we choose the set of sections 
$\x:D \to Y$ to be real. Note that each fiber $Y_t$ 
 comes naturally equipped with a real structure $\tau_t$ when $t\in\R$.
If $\F\cap \R E=\emptyset$, then by perturbing $\F$ if necessary, we
may assume that $\overline f(\overline C)\cap \F\cap E=\emptyset$ for
all $\overline f\in \CC(d,\x(0),J_0)$.

Theorem \ref{mainthm} is obtained by choosing $\x$ such that
$\x(0)\cap X_0=\emptyset$.

\begin{proof}[Proof of Theorem \ref{mainthm}(1)]
Assume that $\x(0)\cap X_0=\emptyset$ and $ \x(0)\cap \R X_1\subset L_1$, and
let us choose a real element   $\overline f:\overline C\to X_\sharp$ 
of $\CC(d,\x(0),J_0)$. Denote
 by $a$ (resp. $b$) the number of real
(resp. pairs of $\tau_0$-conjugated) points of intersections of $E$ and 
$\overline f(\overline C)$.
The map $\overline f:C\to X$ is
real, and by Corollary \ref{lem:lem2} the set
$\CC_{\overline f}$ has
exactly
\begin{itemize}
\item $\sum_{k=a_i+2b_i}{a \choose
  a_i}{b\choose b_i}$ real elements if $\tau_0$ acts
  trivially on $H_2(X_0;\Z)$;

\item $2^k$ if $a=0$ and $b=k$, and $0$ otherwise, real
  elements if $\tau_0$ exchanges $l_1$ and $l_2$.
\end{itemize}

By assumption $(H_1)$, we have that $\F_0\cup L_0$ represents a cycle
$\gamma V$ in $H_2(X_0;\Z/2\Z)$ with $\gamma=0,1$, so
in both cases above we  have 
$m_{L_\sharp,\F_0}(\overline f(\overline C))= m_{L_1,\F_\sharp}(\overline f_{| C_1})+ \gamma
(a+b)$. By Proposition \ref{prop:degeneration}, any element $\overline f$ of
$\CC(d,\x(0),J_0)$ is the limit of a  unique element of 
$\CC(d,\x(t),J_t)$, so this latter has to be real when $\overline f$
is real and $t\in \R^*$.
Hence  to end the proof of Theorem \ref{mainthm}(1), it remains to
show that no node appears in a neighborhood of $E\cap \overline
f(\overline C)$ when deforming $\overline f$.
This follows from the description provided at the end of the
proof of Proposition \ref{prop:degeneration} of the local deformation
of $\overline f$ (since $s=1$ in the present case). An alternative
proof is to observe that $\overline f(\overline C)$ has as many nodes
as any of its deformation:
\[\begin{aligned}
&\frac{(d-k[E])^2 -c_1(X_1)\cdot (d-k[E]) +2}{2} + k(d\cdot[E] +k)
 &= &  \frac{d^2 -c_1(Y_t)\cdot d  +2 }{2}
\end{aligned}
\]
since $[E]^2=-2$ and $c_1(X_1)\cdot [E]=0$.
\end{proof}

\medskip
\begin{proof}[Proof of Theorem \ref{mainthm}(2)]
Assume again that $\x(0)\cap X_0=\emptyset$ and $\x(0)\cap \R X_1\subset L_1$.
Recall that by assumption $(H_2)$ we have
 $\tau_0=\tau_{el}$, which implies in particular that $a=0$.
Suppose that $b\ne 0$, and choose a pair $\{p,\tau_0(p)\}$ of 
$\tau_0$-conjugated intersection points of $E$ with $\overline f(\overline C)$.
Let $\overline f':\overline C'\to X_\sharp$ be 
an element of
$\CC_{\overline f}$, and denote by $\overline C_p$ (resp. $\overline
C_{\tau_0(p)}$) the irreducible component of $C_0$
 whose image contains $p$ (resp. $\tau_0(p)$).
Define the map 
 $\overline
f'':\overline C'\to X_\sharp$ 
as follows: $\overline f''(x)=\overline f'(x)$ if $x\notin \overline  C_p\cup \overline
C_{\tau_0(p)}$, and  
$\overline f''(x)=\tau_0\circ \overline f'(x)$ if $x\in \overline  C_p\cup \overline
C_{\tau_0(p)}$.
The map $\overline f''$ is also an element of $\CC_{\overline f}$, and 
 it follows from Lemma \ref{lem:lem3} below that 
$m_{L_\sharp,\F_\sharp}(\overline f'(\overline C'))=-
 m_{L_\sharp,\F_\sharp}(\overline f''(\overline C'))$. Hence
$$\sum_{\overline f'\in \CC_{\overline f}} m_{L,\F_\sharp}(\overline
 f'(\overline C'))=0,$$
and Theorem \ref{mainthm}(2) is proved.
\end{proof}

Given a point $p\in E$, we denote by $C_p$
 the (unique) $J_{0|X_0}$-holomorphic curve in the class $l_1$
passing through $p$. 
\begin{lemma}\label{lem:lem3}
Let $\D\subset \C P^1\times \C P^1$
 be an embedded $\tau_{el}$-invariant disk with $\partial \D\subset E$.
Then if $p\in E\setminus  \partial \D$, the parity of the number of
intersection points of $\D$ with $C_p$ and 
$C_{\tau_{el}(p)}$ are different.
\end{lemma}
\begin{proof}
Denote by $D_1$ and $D_2$ the two halves of $E\setminus \partial \D$.
According to the proof of Lemma \ref{lem:real disk}, up to exchanging
$D_1$ and $D_2$ we have
$[D_1\cup \D]=l_1$ and $[D_2\cup
  \D] =l_{2}$ in $H_2(\C P^1\times \C P^1;\Z/2\Z)$.
Any
$J_{0|X_0}$-holomorphic curve in the class $l_i$ intersects $E$ in
exactly one point, so the result follows
 from the fact that both
$C_p$ and $C_{\tau_{el}(p)}$ intersect $D_1\cup \D$ in an even number
 of points.
\end{proof}

\subsection{Proof of Theorem  \ref{thm:W vanish}}\label{sec:proof vanish}
We prove Theorem  \ref{thm:W vanish} by choosing the set of sections
$\x$ so that $\x(0)\cap X_0$ is reduced to a single 
point. In this case, it follows from Proposition
\ref{prop:degeneration} that an element $\overline f$ 
of $\CC(d,\x(0),J_0)$ might be
the limit of two distinct elements of $\CC(d,\x(t),J_t)$. Next
proposition is a real version of Proposition \ref{prop:degeneration}
in this case.

\begin{prop}\label{prop:real degeneration}
Suppose that $\x(0)\cap X_0=\{p_0\}$ and  $\x(0)\cap \R X_1\ne\emptyset$. 
Let $\overline f:\overline C\to X_\sharp$ be a real element 
of $\CC(d,\x(0),J_0)$, with a point $p\in C_1$ such that $\overline f(C_1)$ has a tangency with $E$ at
$\overline f(p)$. 
Given $t\ne 0$, let $f_1:\C P^1\to Y_t$ and  $f_2:\C P^1\to Y_t$ be
the two deformations of $\overline f$ in $\CC(d,\x(t),J_t)$ (see Proposition \ref{prop:degeneration}).

Then $p$ is a real point of $\overline C$,
and both $f_1(\C P^1)$ and $f_2(\C P^1)$ have a unique node $q$ arising from
the smoothing of $\overline C$ at $p$.
Moreover, there exists $\varepsilon=\pm 1$ such that
neither $f_1$ nor $f_2$ are real when
$\varepsilon t<0$, and both $f_1$ and $f_2$ are real when $\varepsilon t>0$.  
In this latter case, up to exchanging $f_1$ and $f_2$, we have
(see Figure   \ref{fig:real part 2}):
\begin{itemize}
\item $f_{1}^{-1}(q)\in \R P^1$;
\item $f_{2}^{-1}(q)\notin \R P^1$ and
$f_{2}(\R P^1)\cap U=\emptyset$, where $U$ is the connected
  component that contains $q$ of the intersection of $Y_t$ 
with  a small neighborhood in $\R Y$ of $\overline f(p)$.
\end{itemize}
\end{prop}
 \begin{figure}[h]
\centering
\begin{tabular}{c}
\includegraphics[height=5.5cm, angle=0]{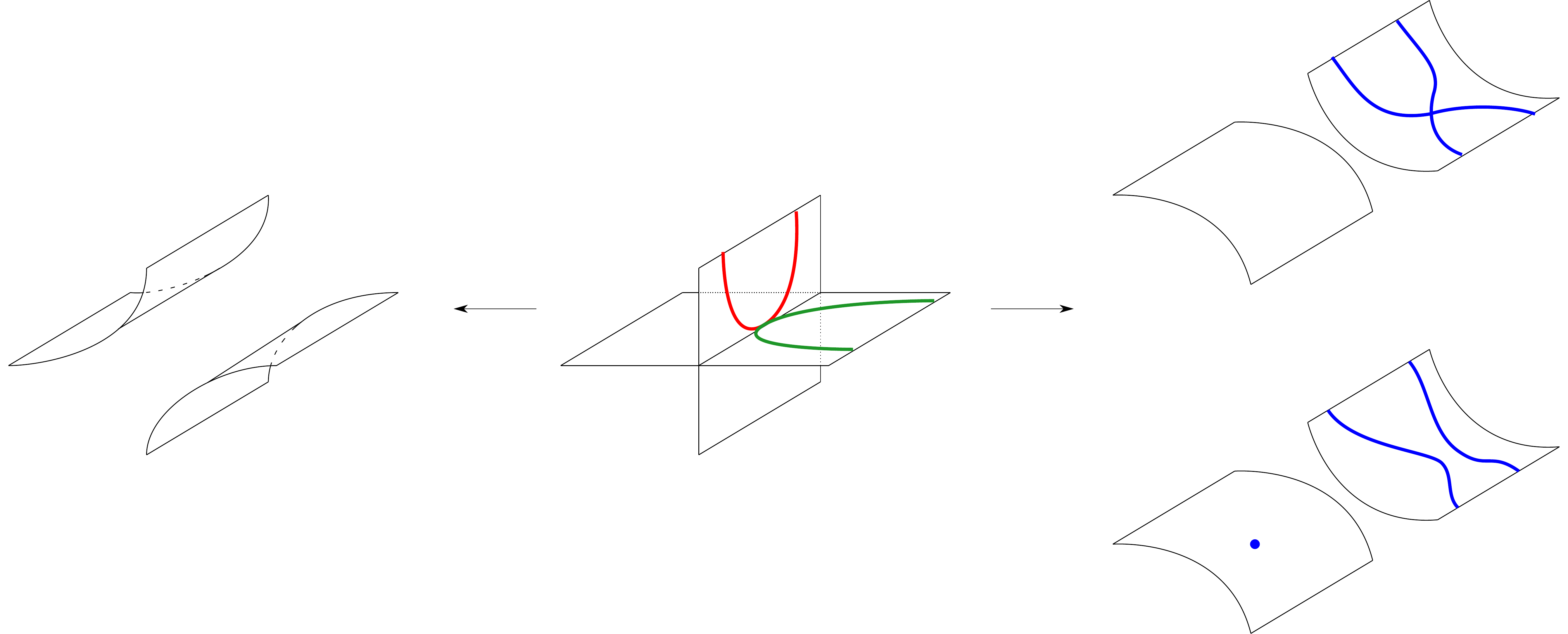}
 \put(-20, 145){$f_1(\R P^1)$}
 \put(-40, 120){$q$}
 \put(-20, 60){$f_2(\R P^1)$}
 \put(-80, 30){$q$}
 \put(-410, -30){$\varepsilon t<0$, no real deformation}
 \put(-120, -30){$\varepsilon t>0$, two real deformations}
\\ 
\end{tabular}
\caption{Real deformations of a real map 
$\overline f:\overline C\to X_\sharp$ which is the limit of two maps}
\label{fig:real part 2}
\end{figure}
\begin{proof}
It follows from Proposition \ref{prop:degeneration} that the point $p$
is unique, and hence real. Since $\overline f(\overline C)$ has one
node less that any of its deformation, we deduce that  both $f_1(\C
P^1)$ and $f_2(\C P^1)$ have a unique node $q$ arising from
the smoothing of $\overline C$ at $p$. Since $q$ is unique, it has to
be real if the deformation is real.

Recall from the end of the proof of Proposition
\ref{prop:degeneration} how looks like a deformation of $\overline f$
in a neighborhood of the smoothing of $p$.
The manifold
$Y$ is given in local coordinates $(t,x,y)$ at $\overline f(p)$
by the equation $xy=t$, the manifold $X_0$
(resp. $X_1$) being
locally given by $\{t=0\ \mbox{and}\ y=0\}$ 
(resp. $\{t=0\ \mbox{and}\ x=0\}$). Furthermore 
 the maps $\overline f_{C_0}$ and
$\overline f_{C_1}$ have expansions
$$ x(z)=az^2 + o(z^2)\quad\mbox{and}\quad y(w)=bw^2 +o(w^2), $$
where $a,b\in \R^*$, and $z$ and $w$ are local coordinates at $p$ of $C_0$ and
$C_1$ respectively. 
For $0<|t|<<1$, the two maps $f_1$ and $f_2$ correspond to the two solutions
 $\mu(t)\in\C^*$  of
$$\mu(t)^2=\frac{t}{ab}.$$
For each solution, the smoothing of $\overline C$ at $p$ is locally given by
$zw=\mu(t)$, and the corresponding deformation is approximated by the map
$$g_t:\{zw=\mu(t)\}\subset \C^2\mapsto (t,az^2,bw^2) $$
close to the smoothing of $p$.

If $tab<0$, then the two solutions of 
$\mu(t)^2=\frac{t}{ab}$ are complex conjugated, and neither $f_1$ nor
$f_2$ are real.
On the opposite, if $tab>0$, then the two solutions of 
$\mu(t)^2=\frac{t}{ab}$ are real, and both $f_1$ and $f_2$ are real.
 Moreover  the arcs of $\R C_0\setminus\{p\}$ and 
$\R C_1\setminus\{p\}$ are glued in a different way 
for $f_{1}$ and $ f_{2}$ (see Figure \ref{fig:real part 3}). 
In particular  one of them, say 
$ f_{1}$, satisfies $f_{1}^{-1}(q)\in \R P^1$, while 
$f_{2}$ satisfies $f_{2}^{-1}(q)\notin \R P^1$.
 \begin{figure}[h!]
\centering
\begin{tabular}{c}
\includegraphics[height=5.5cm, angle=0]{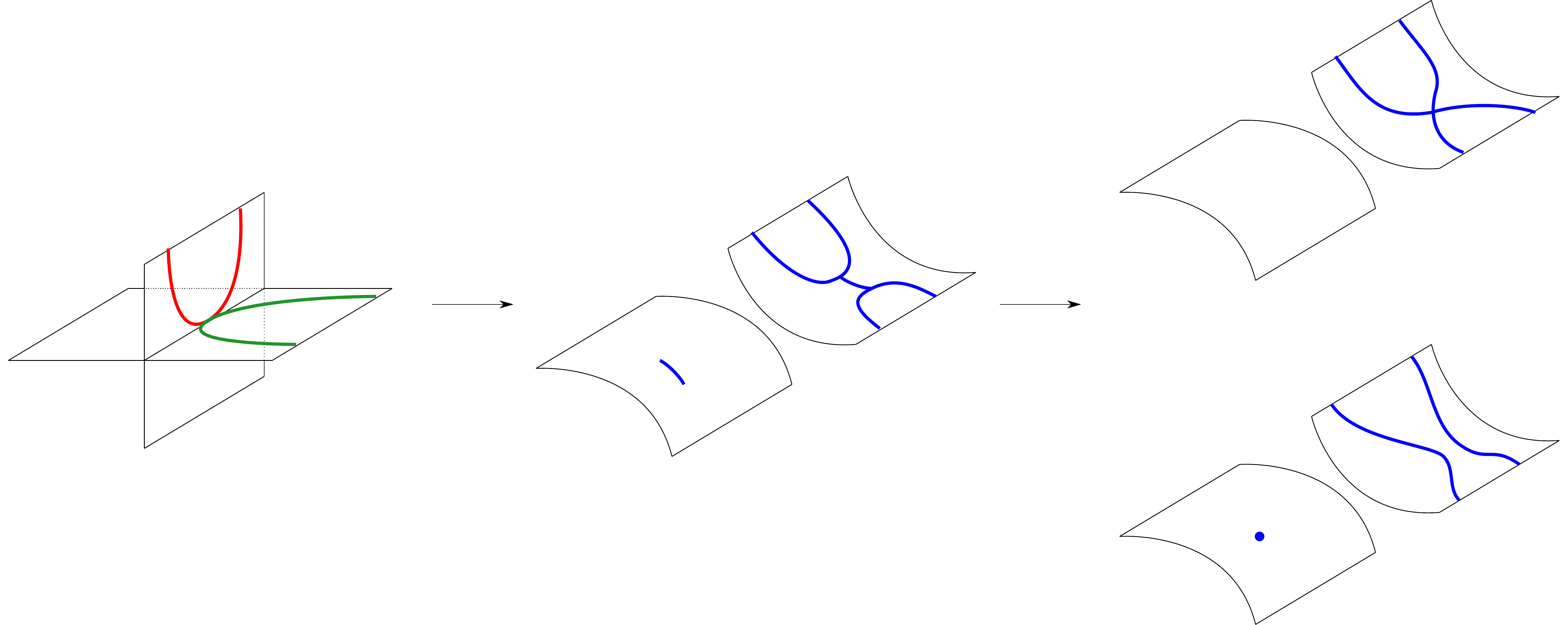}
 \put(-255, 20){Approximation by $g_t$}
\\ 
\end{tabular}
\caption{Real deformations of
$\overline f:\overline C\to X_\sharp$,
  intermediate step}
\label{fig:real part 3}
\end{figure}

Let $U$ be the connected
  component that contains $q$ of the intersection of $Y_t$ 
with  a small neighborhood in $\R Y$ of $\overline f(p)$. We have
to prove that $ f_{2}(\R P^1)\cap U=\emptyset$.
Suppose that this is not the case, and let $S\subset Y_t$ be a topological surface passing
through $q$, and locally a cylinder in the variable $t$ at $q$. 
Then the set $ f_{2}^{-1}(S)$ would contain four points in a
neighborhood of a smoothing of $p$. However  
the set $ g_t^{-1}(S)$ has only two points in
$\{zw=\mu(t)\}$, which contradicts the fact that $ f_{2}$
is approximated by $g_t$ close to the smoothing of $p$.
\end{proof}

Now we are ready to prove  Theorem  \ref{thm:W vanish}.
Recall that  by assumption $L_0$ is a disk, which in particular
implies that  $\tau_0=\tau_{el}$.

\begin{proof}[Proof of Theorem \ref{thm:W vanish}(1)]
Suppose that $\x(0)\cap X_0=\{p_0\}$, and  $\x(0)\cap
\R X_1$ is non-empty and contained in $L_1$. Without loss of generality,
we may assume that $\x(t)\cap \R Y_t \subset L$ when $t>0$ (and so
$\x(t)\cap \R Y_t \not\subset L$ when $t<0$, since $L'$ contains the
deformation of $p_0$).
A schematic picture of the
degeneration of $\R Y_t$ to $\R X_\sharp$ is provided in Figure \ref{fig:degen 1}. 
Denote by $\overline L$ the connected component of $\R X_1$ containing $L_1$.
Since $L_0$ is a disk, we necessarily have $L'\ne L$,
and $\overline L \setminus \R E$ is disconnected.
 \begin{figure}[h!]
\centering
\begin{tabular}{ccc}
\includegraphics[height=3cm, angle=0]{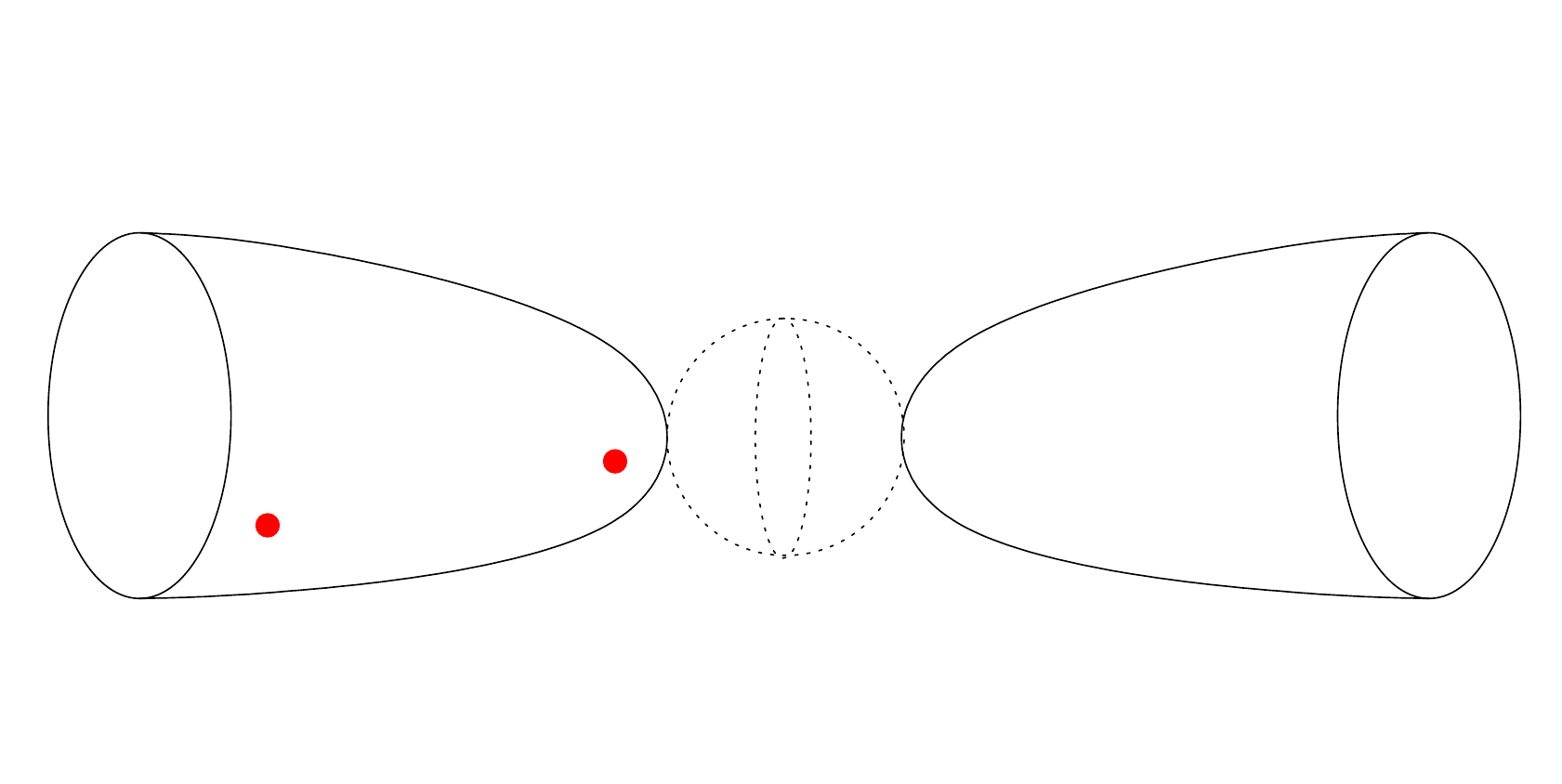}&
\put(-150,10){$L$}
\put(-95, 10){$S_V$}
\put(-50, 10){$L'$}
\put(-105, -20){$\R Y_t$, $t>0$}
\hspace{9ex} &
 \includegraphics[height=3cm, angle=0]{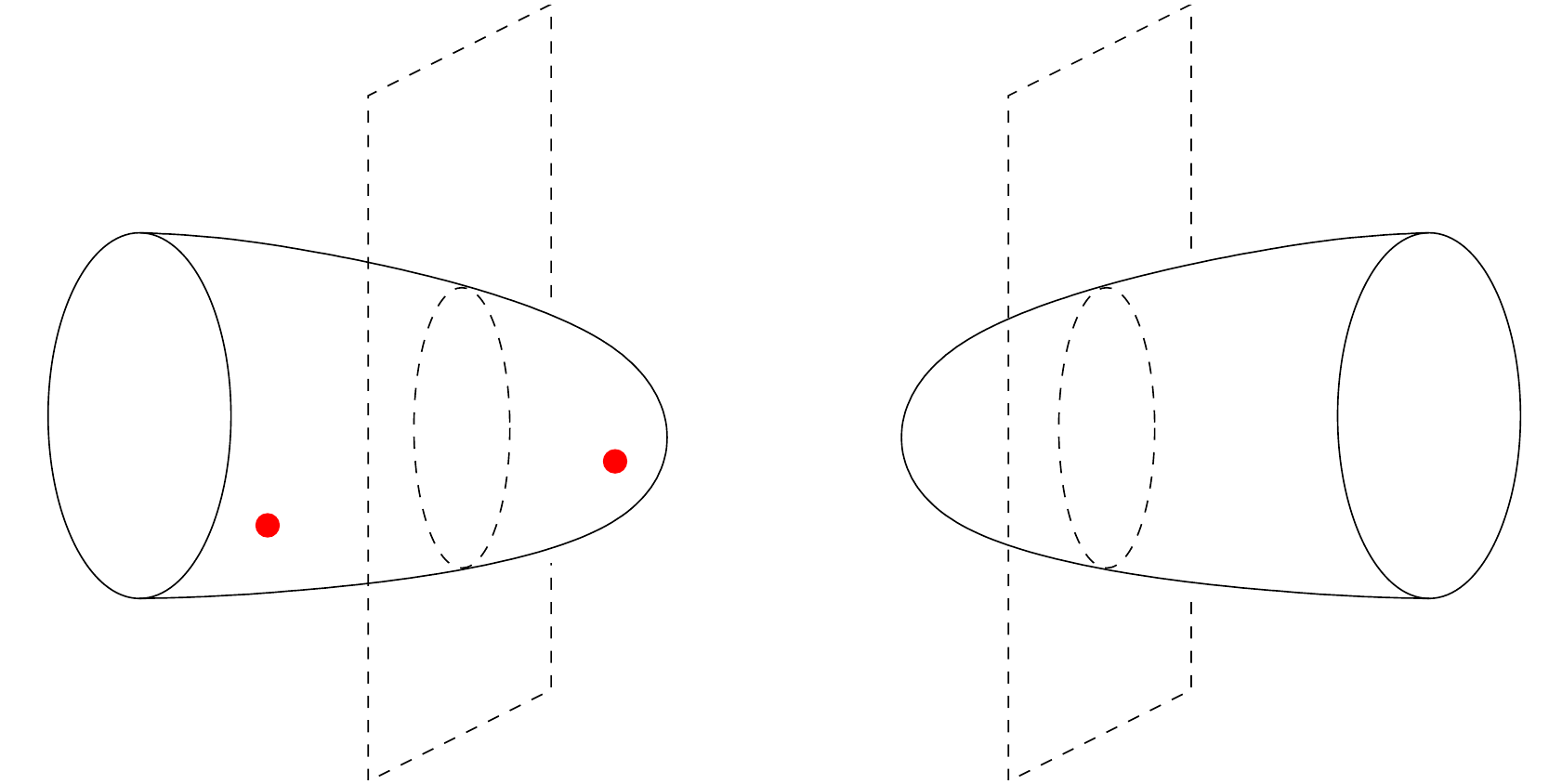}
\put(-95, -0){$\R X_0$}
\put(-165, -0){$\R X_1$}
\put(-35, -0){$\R X_1$}
\put(-95, -20){$\R X_\sharp$}
\\ 
\end{tabular}
\caption{Degeneration of $\R Y_t$ to $\R X_\sharp$}
\label{fig:degen 1}
\end{figure}

Let  $\overline f:\overline C\to X_\sharp$ be a real  element
of $\CC(d,\x(0),J_0)$. Recall that $\overline C'$ denotes the
irreducible component of $C_0$ whose image passes through
the point $p_0$.

Suppose first that $\overline f_*[\overline C']= l_i$. Since
$\tau_{el}$ exchanges $ l_1$ and $l_2$, there exists an irreducible
component $\overline C''$   of $C_0$ such that 
$\tau_{el}\circ \overline f(\overline C'')= \overline f(\overline C')$. 
However $\overline f(\overline C'')\cap \overline f(\overline
C')=\{p_0\}$, which contradicts that $J_0$ is generic.

Hence  $\overline f_*[\overline C']= l_1+l_2$ for any choice of
$\overline f$. Since $\tau_0=\tau_{el}$ and any other component of
$ C_0$ realizes a class $l_i$, the curve $\overline C'$ is
the only real component of $C_0$, and $\overline f^{\ -1}(\R
E)$ consists of at most  $2$ points. 
Suppose that $\overline f_{|\overline C'}\in \CC^{2e_1,
  0}(l_1+l_2,\{p_0\}\cup \x_E, J_0)$. 
In particular, 
 $C_1$ has two irreducible components $\overline C_{1}$
and
$\overline C_{2}$. 
Since $\x(0)\cap
\R X_1\ne\emptyset$, both  $\overline C_{1}$
and
$\overline C_{2}$ must be real with
a non-empty real part. Since $\overline C$ has arithmetic genus $0$, 
we deduce that $\overline f^{\ -1}(\R E)$ consists of precisely $2$
points, which are the intersection points of $\overline C'$ with 
$C_1$. Hence both  $\overline f(\R  \overline C_{1})$
and  $\overline f(\R\overline  C_{2})$ intersect $\R E$ in exactly
one point, where this intersection is transverse. 
But this contradicts the fact that $\overline L \setminus \R E$ is disconnected.

Hence  $\overline f_{|\overline C'}\in \CC^{e_2,
  0}(l_1+l_2,\{p_0\}\cup \x_E, J_0)$ for any choice of
$\overline f$, 
and Theorem \ref{thm:W vanish}(1) is now a consequence of
 Proposition \ref{prop:real degeneration} (see Figure \ref{fig:degen 2}).
 \begin{figure}[h!]
\centering
\begin{tabular}{ccc}
\includegraphics[height=2.5cm, angle=0]{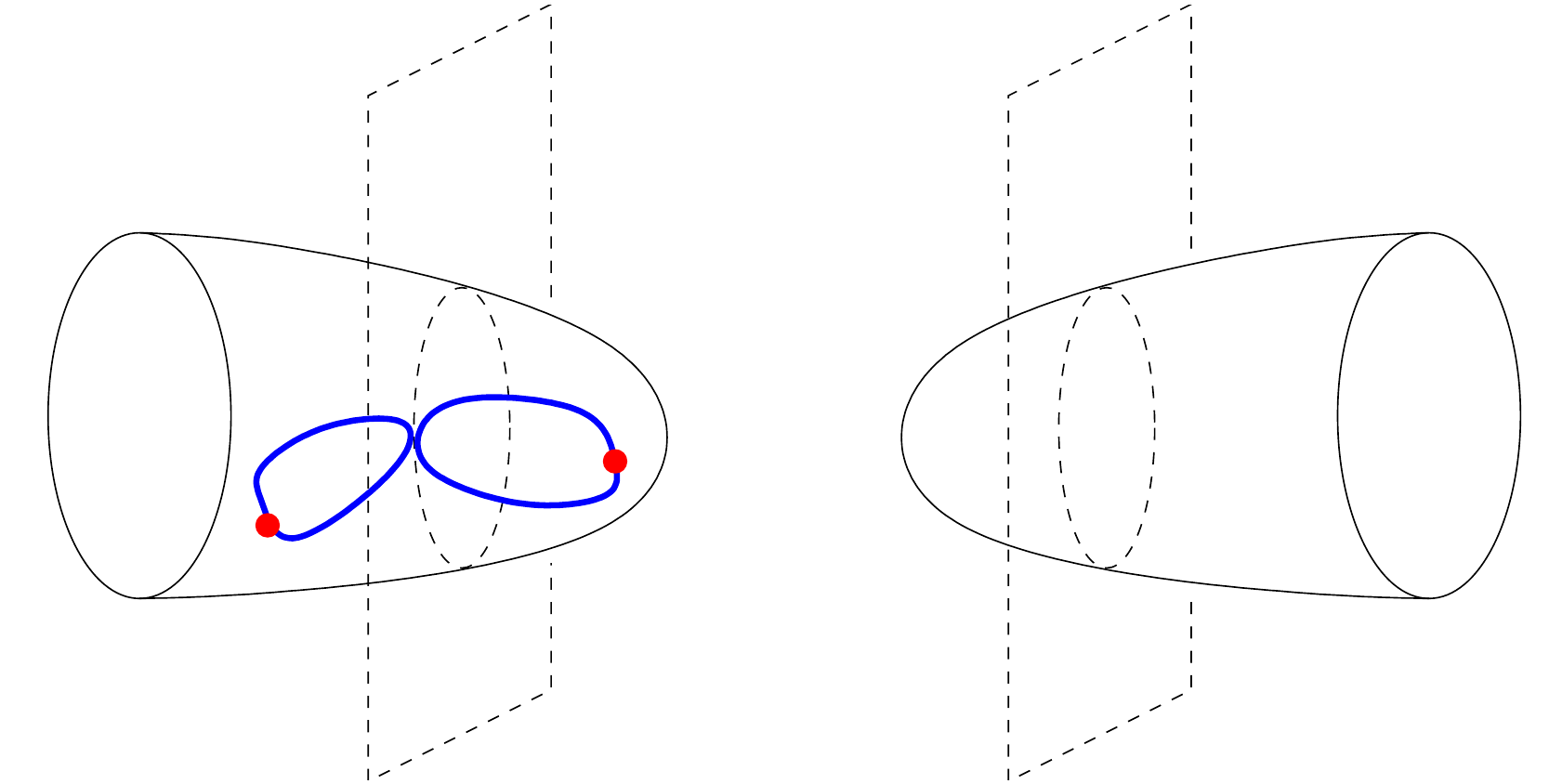}&
\includegraphics[height=2.5cm, angle=0]{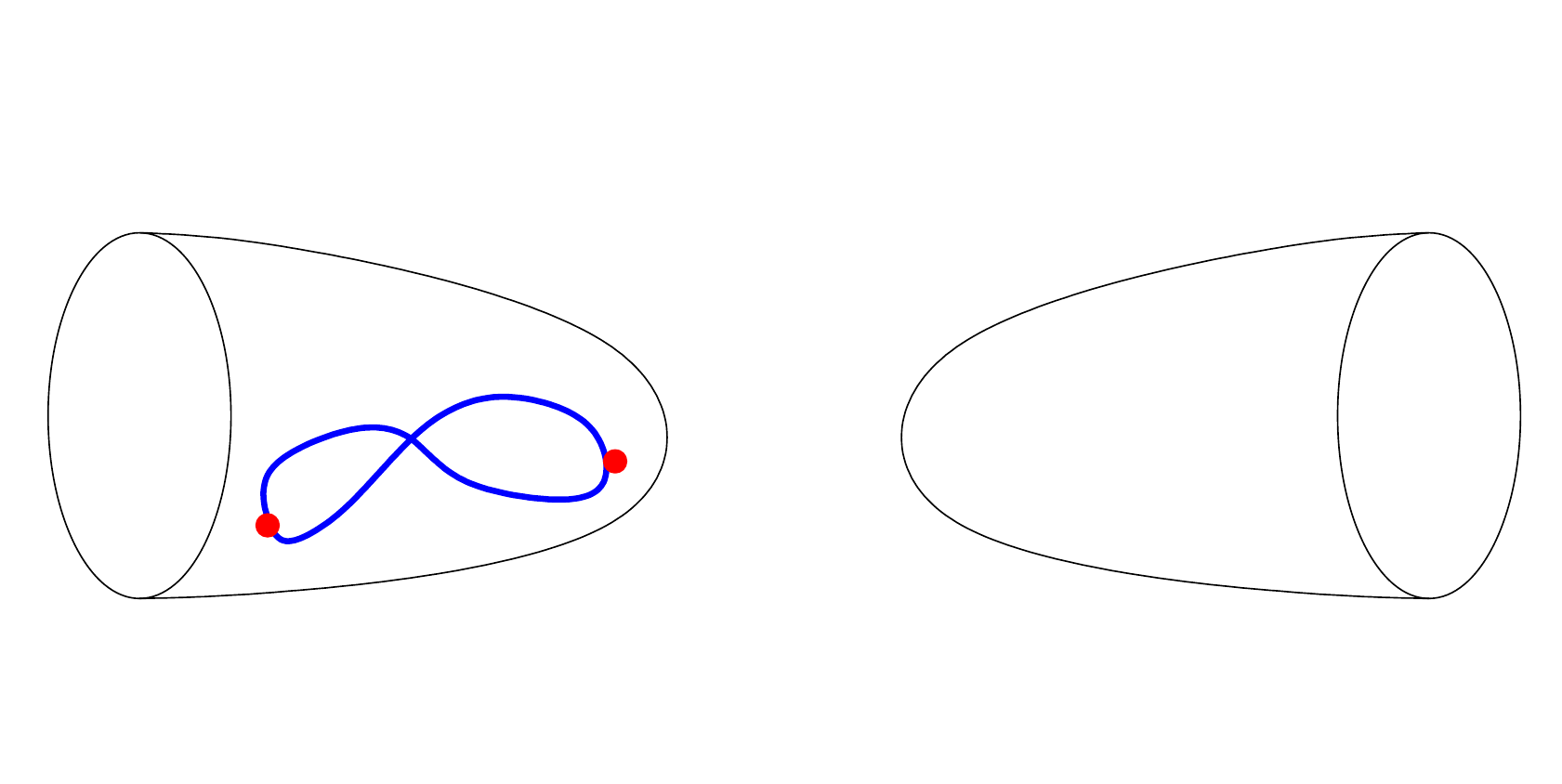}&
 \includegraphics[height=2.5cm, angle=0]{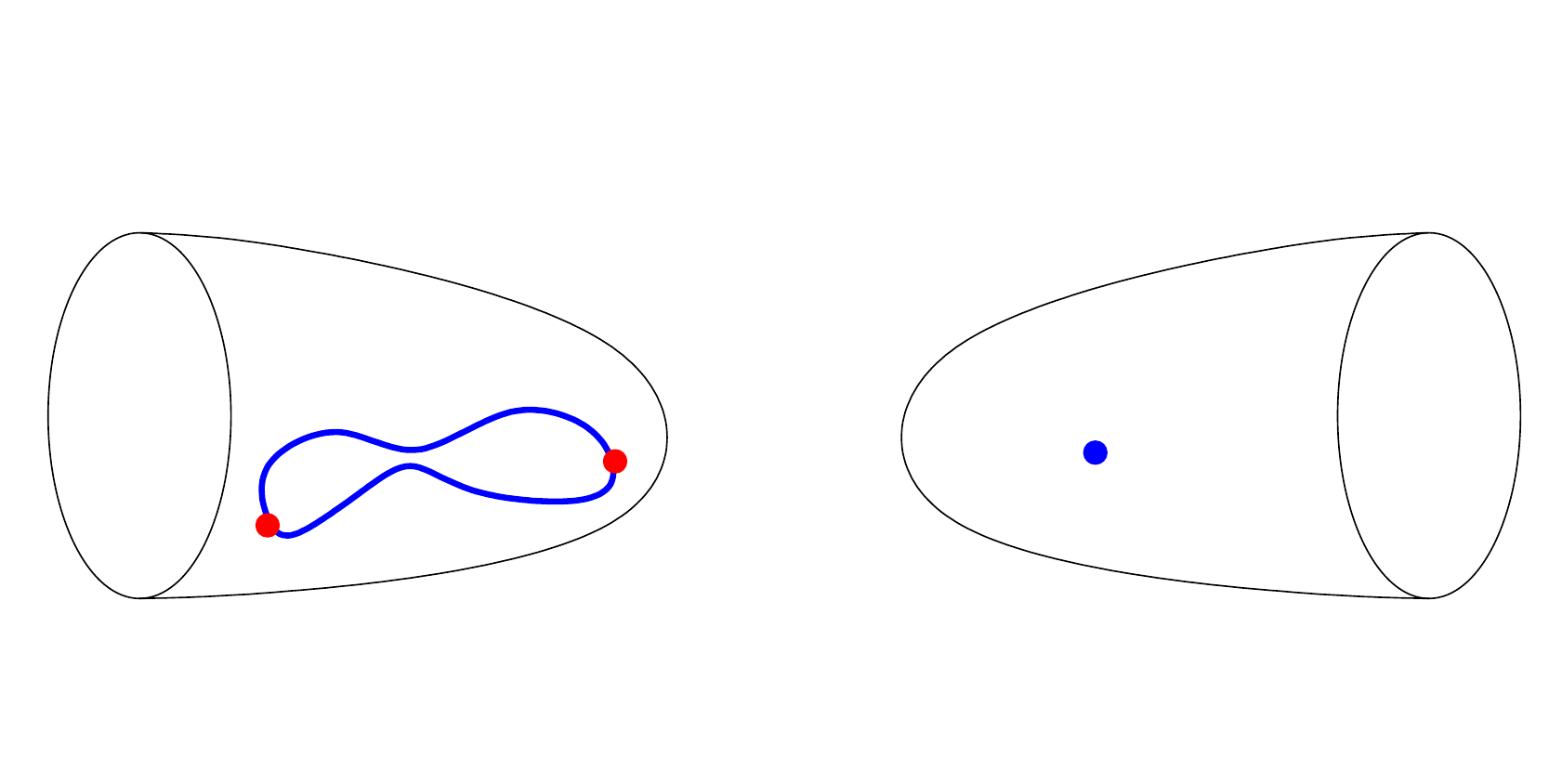}
\\ 
\end{tabular}
\caption{$\overline f(\overline C)$ and its two real deformations}
\label{fig:degen 2}
\end{figure}
\end{proof}

\begin{proof}[Proof of Theorem \ref{thm:W vanish}(2)]
Assume now that that $\x(0)\cap X_0=\{p_0\}$ and  $\x(0)\cap
\R X_1=\emptyset$.
According to the proof of Theorem \ref{thm:W vanish}(1), the only
maps $\overline f\in\CC(d,\x(0),J_0)$ with a non-trivial contribution
to $W_{X_\R,L,F}(d,s)$ satisfy $\overline f_{|\overline C'}\in \CC^{2e_1,
  0}(l_1+l_2,\{p_0\}\cup \x_E, J_0)$. In particular, the curve $C_1$
has two  irreducible components, which are exchanged by the real
structure on $\overline C$. There are $2^{\frac{c_1(X)\cdot d -4}{2}}$ ways of
distributing the points in $\x(0)\cap X_1$ among these two
components, which proves the result about divisibility of 
$W_{X_\R,L,F}(d,s)$.

Moreover $\overline C'$ is the only real irreducible component of $\overline
C$ and the map $\overline f_{|\overline C'}$ is an embedding. The
adjunction formula implies that
the number of  real solitary nodes of $\overline f(\overline C)$  has the same parity than $\frac{d^2-c_1(X)\cdot d+2}{2}$.
\end{proof}

\section{Real algebraic rational
  surfaces}\label{sec:real rational}

Here we deduce Theorem \ref{thm:real algebraic} from Theorem
\ref{thm:chain spheres} and the classification of
real rational algebraic surfaces (see for example \cite{Sil89,Kol97}).
The proof goes by explicit computations of homology groups and
direct application of Theorem \ref{thm:chain spheres}. 
Recall that any real algebraic minimal rational surface with a non
empty real part corresponds to  exactly one of the
following cases:
\begin{itemize}
\item $\C P^1\times \C P^1$  equipped with the real structure $\tau_{el}$;
\item $\C P^2$ equipped with the complex conjugation;
\item minimal conic bundles;
\item covering of degree 2 of $\C P^2$ ramified along a maximal real quartic;
\item covering of degree 2 of the quadratic cone in $\C P^3$ 
 ramified along a maximal real cubic section.
\end{itemize}
We treat all these  cases in Sections \ref{sec:H vanish},
\ref{sec:conic bundle}, \ref{sec:quartic}, and \ref{sec:cone}, and
prove Theorem \ref{thm:real algebraic} in Section \ref{sec:proof}.

\subsection{Generalities}
In this section, we fix once for all a
real 
rational 
symplectic $4$-manifold
 $X_\R=(X,\omega,\tau)$ and $L$
a connected component of $\R X$.
Since $(X, \omega)$ is diffeomorphic to either $\C P^1\times \C P^1$
or to  $\C P^2$ blown-up at finitely many points, 
all homology groups of
$X$ are known, and the intersection form on $H_2(X;\Z/2\Z)$ is non-degenerate.

\begin{lemma}\label{lem:MV}
The following hold:
\begin{itemize}
\item $b_2(X\setminus L;\Z/2\Z)=b_2(X;\Z/2\Z)+b_1(L)+
b_1(X\setminus L)-1;$
\item the group  $H_1(L;\Z/2\Z)$ is naturally isomorphic to the
kernel of the natural map $\iota:H_2(X\setminus L;\Z/2\Z)\to
H_2(X;\Z/2\Z)$;
\item  $b_1(X\setminus L)=0$ if $[L]\ne 0$ in $H_2(X;\Z/2\Z)$,
and $b_1(X\setminus L)=1$ otherwise.
\end{itemize}
\end{lemma}
\begin{proof}
Let $U$ be a tubular neighborhood of $L$ in
$X$ (in particular $U$ retracts to $L$). Since $X$ is simply connected, 
the Mayer-Vietoris sequence applied to $X= (X\setminus L)\cup U$
gives the exact sequence
$$0\xrightarrow[]{}  H_2(U\setminus L;\Z/2\Z)\xrightarrow[]{(i_2,j_2)}  
H_2(L;\Z/2\Z) \oplus H_2(X\setminus L;\Z/2\Z) 
\xrightarrow[]{}  H_2(X;\Z/2\Z)\xrightarrow[]{\partial}  $$
\begin{equation}\label{eq:MV}
\xrightarrow[]{\partial}
H_1(U\setminus L;\Z/2\Z)\xrightarrow[]{(i_1,j_1)}  
H_1(L;\Z/2\Z) \oplus H_1(X\setminus L;\Z/2\Z) 
\xrightarrow[]{}  0 .
\end{equation}
The space $U\setminus L$ retracts to an $S^1$-bundle $\psi:M\to L$ over $L$,
hence it follows 
from Poincaré duality that $b_2(U\setminus L;\Z/2\Z)=b_1(U\setminus
L;\Z/2\Z)$.
Together with the exact sequence $(\ref{eq:MV})$, this implies that
$$b_2(X\setminus L;\Z/2\Z)=b_2(X;\Z/2\Z)+b_1(L)+
b_1(X\setminus L)-1.$$
Each loop $\gamma$ in $L$ produces a surface $\psi^{-1}(\gamma)$ in $M$. By the Gysin
 sequence, this  induces an injective map
 $\kappa:H_1(L;\Z/2\Z)\hookrightarrow H_2(M;\Z/2\Z)$, and 
we have (intersection numbers are in $\Z/2\Z$)
$$b_2(M;\Z/2\Z)=b_1( L;\Z/2\Z)+1 - L^2.$$
The map $\psi_*:H_2(M;\Z/2\Z)\to H_2( L;\Z/2\Z)$ admits a section if and only if
 $L^2=0$. In this case the extra generator of $H_2(M;\Z/2\Z)$ is precisely
given by the image of such a section.
By definition of the Mayer-Vietoris sequence $(\ref{eq:MV})$, we
obtain that 
$$\mbox{Ker }\iota\simeq\mbox{Ker }i_{2}=\mbox{Im }\kappa\simeq  H_1(L;\Z/2\Z).$$

Analogously, the natural map  $\psi_*:H_1(M;\Z/2\Z)\rightarrow H_1(L;\Z/2\Z)$
is surjective with kernel generated by the class $\nu$ realized by
 a fiber of $\psi$, and $\nu=0$  if and only if
$L^2=1$. By definition of the Mayer-Vietoris sequence
 $(\ref{eq:MV})$, the same holds for the map $i_1$. We deduce
  that
$b_1(X\setminus L)=1-\mbox{rank } \partial$.
If $S$ is a closed surface in $X$ intersecting $L$ transversely in
finitely many points $p_1,\ldots,p_k$, we have
$$\partial( [S])= [\psi^{-1}(p_1)] +\ldots+ [\psi^{-1}(p_k)]=
([S]\cdot [L])\nu.   $$
Hence the map $\partial$ is null if and only if $[L]$ is in the kernel
of the intersection form on $H_2(X;\Z/2\Z)$. This intersection
form in non-degenerate, hence  the map $\partial$ is null if and only
if $[L]=0$. 
\end{proof}

 The consideration of the group $\mathcal H(X_\R, L)$ 
is justified by the  next two propositions.
\begin{prop}\label{prop:curve}
Let 
$\delta$ be a
$\tau$-invariant class in
$H_2(X\setminus L;\Z/2\Z)$  realized by a 
smooth
real symplectic
curve $E$. Assume in addition that $\delta^2\ge -1$ if $E$ is a sphere. Then for any $d\in H_2(X;\Z)$, we have 
$$W_{X_\R,L,F}(d,s)= (-1)^{\frac{d\cdot \delta}{2}}\ W_{X_\R,L,F+\delta}(d,s).$$
\end{prop}
\begin{proof}
Choose a configuration $\x$ made of $c_1(X)\cdot d-1 -2s$ points in $L$ and  $s$ 
 pairs of $\tau$-conjugated 
points in $X\setminus \R X$.
Let $J_0$ be a generic $\tau$-compatible almost complex structure on $X$ such
that $E$ is $J_0$-holomorphic.
By the same arguments used in the proof of Proposition
\ref{prop:finite}, if $f:C\to X$ is a $J_0$-holomorphic map from a
nodal curve of arithmetic genus $0$, and such that $f_*[C]=d$ and
$\x\subset f(C)$, then $C$ is actually smooth and irreducible.
Furthermore  all intersection points of $f(C)$ and $E$ are positive, 
so the intersection $E\cap f(C)$ is made of $d\cdot \delta$
distinct points if $f(C)\not\subset E$.
If $f$ is in addition real 
and such
that $L$ contains $f(\R C)$, 
since  both curves $E$ and $f(C)$ are real with disjoint real parts, we have 
that 
$\underline{f(C)}\cdot \delta={\frac{d\cdot \delta}{2}}$, and this
equality is preserved modulo 2
under deformation of both $f$ and $J_0$.
\end{proof}

Assume now that $X_\R$ is a real algebraic rational surface.
The real part of a real  symplectic curve $C$ 
in $X$ defines a class $l_C$ in $H_1(X;\Z/2\Z)$.
It follows from the classification of real rational algebraic
surfaces that any class in $H_1(\R X;\Z/2\Z)$ is realizable
 by a real deformation of a real algebraic curve. 
If two real  cycles in $X$ intersect in finitely many points, 
the parity of this number 
 only depends on the classes realized by these cycles in 
$H_2(X;\Z/2\Z)$. Moreover the intersection form modulo 2 is
non-degenerated on $H_1(\R X;\Z/2\Z)$. Hence 
the class $l_C$ only depends on
$[C]\in H_2(X;\Z/2\Z)$, and we denote it by
$l_{[C]}$.

\begin{prop}\label{prop:vanish}
Let 
$\delta$ an element of $ Ker\ \iota\simeq   H_1(L;\Z/2\Z)$. 
Then for any $d\in H_2(X;\Z)$, we have 
$$W_{X_\R,L,F}(d,s)= (-1)^{\delta\cdot l_d}\ W_{X_\R,L,F+\delta}(d,s).$$
\end{prop}
\begin{proof}
Let $C$ be a real  symplectic curve in $X$.
Recall that $\delta$ can be represented by the restriction over a loop
$\gamma$ of the boundary of a tubular neighborhood of $L$ in
$X$. We denote by $\gamma'$ this representative of $\delta$.
Without loss of generality, we may further assume that
$\gamma$ intersect $\R C$ transversely  and in finitely many points.
Note that the tubular neighborhood of $L$ in $X$ can be chosen as
small as needed.
In particular, all intersection
points of $\gamma'$ and $C$ are located in a neighborhood of $\R
C\cap\gamma$, and each such point corresponds to a pair of
$\tau$-conjugated points of $\gamma'\cap C$.
\end{proof}

\subsection{Surfaces with $\mathcal H(X_\R, L) =0$}\label{sec:H vanish}
We start by giving the list of real algebraic minimal rational
surfaces whose group  $\mathcal H(X_\R, L)$ vanishes. There are
exactly four of them.

\begin{lemma}\label{lem:H vanish}
If $X_\R$ is either $(\C P^1\times \C P^1,\tau_{el})$, or
$(\C P^2,conj)$, or a minimal conic bundle with 
a connected real part, then 
$\mathcal H(X_\R, L) =0$.
\end{lemma}
\begin{proof}
One computes easily that $H_2((\C P^1\times \C P^1)\setminus
S^2;\Z/2\Z)$ (resp. $H_2(\C P^2\setminus
\R P^2;\Z/2\Z)$)
is generated by the class realized by a real algebraic curve in the
class $l_1+l_2$ (resp. a real  conic) with an empty real part.
There exist two minimal conic bundles with a connected real part,
namely $X_\R=(\C P^1\times \C P^1,\tau_{hy})$, and a minimal conic bundle
with $\R X=S^2$. This latter case is covered by
Section
\ref{sec:conic bundle}, so assume that $X_\R=(\C P^1\times \C P^1,\tau_{hy})$.
By Lemma \ref{lem:MV}, we have that
$b_2(X\setminus \R X;\Z/2\Z)= b_2(X;\Z/2\Z)+2$, and that
the kernel of the
 natural map $\iota:H_2(X\setminus \R X;\Z/2\Z)\to
H_2(X;\Z/2\Z)$ is of dimension 2. Hence 
$H_2(X\setminus
\R X;\Z/2\Z)$ is  isomorphic to 
 $\mbox{Ker} \ \iota\times H_2(X;\Z/2\Z)$.
Any class in $H_2(X;\Z/2\Z)$ is realized by a non-singular real rational algebraic curve
with a non-empty real part,
and has intersection number 1 with some other class in
$H_2(X;\Z/2\Z)$. This implies that $\mbox{Ker} \ \iota$ is the set of
$\tau_{hyp}$-invariant classes in $H_2(X\setminus \R X;\Z/2\Z)$,
and the lemma is proved.
\end{proof}

\subsection{Minimal conic bundles}\label{sec:conic bundle}
Let  $(X,\tau)$  be a minimal conic bundle whose real part is
made of $n\ge 1$ spheres. 
Up to real deformation, we may assume that $X$ has the following
affine equation in $\C^3$:
$$y^2+z^2=\prod_{i=1}^{2n}(x- a_i)$$
where $a_1<a_2\ldots<a_{2n}$ are distinct real numbers, and 
$\tau$ is the restriction of the complex conjugation on $\C^3$.
Forgetting the
$(y,z)$-coordinates provides a real projection  $\rho:X\to \C P^1$.
Given $i=1,\ldots n$, we
 denote by $S_{2i-1}$ (resp. $S_{2i}$) the Lagrangian sphere
 $\rho^{-1}([a_{2i-2};a_{2i-1}])\cap 
\R^3$ (resp. $\rho^{-1}([a_{2i-1};a_{2i}])\cap 
\R\times (i\R)^2$), with
  the obvious convention that $a_{2n+1}=a_0$, see Figure \ref{fig:conic}.
 \begin{figure}[h]
\centering
\begin{tabular}{c}
\includegraphics[height=2.6cm, angle=0]{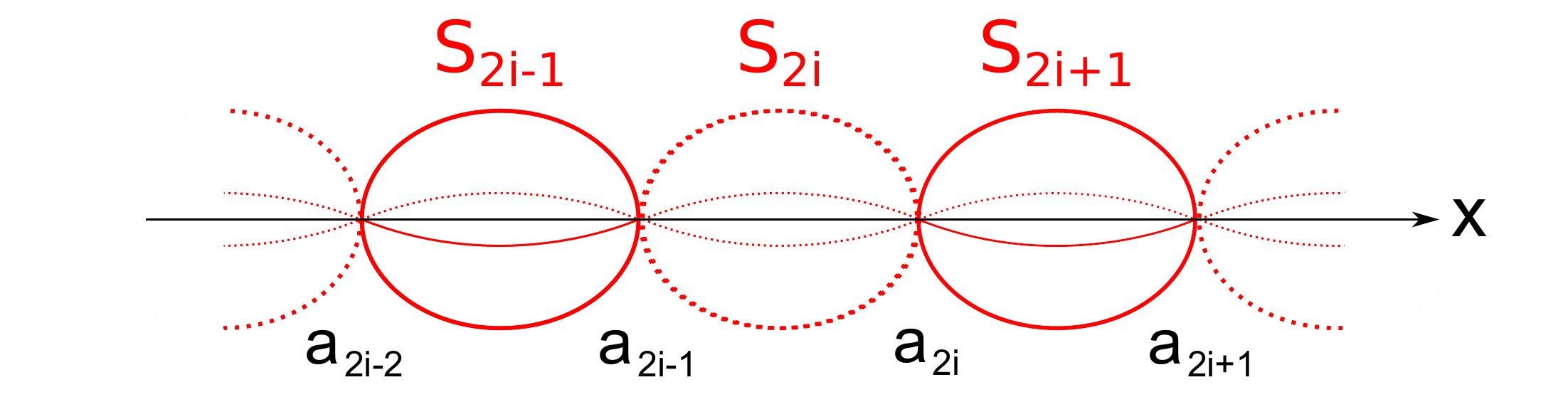}
\end{tabular}
\caption{Real vanishing cycles of Conic bundles}
\label{fig:conic}
\end{figure}
We also denote by  $F$ a generic fiber,
by $E_2$ an
irreducible component  of the singular fiber
$\rho^{-1}(a_2)$, and by $B$ a (non-real) section of $\rho$ which does
not intersect the curve $E_2$.
The real Picard group of $X$ is the free abelian group generated by
$F$ and $c_1(X)$ (see \cite{Sil89,Kol97}).
\begin{lemma}\label{lem:homology conic}
A basis of $\mathcal H(X_\R,S_1)$ is given by
$([S_3],\ldots,[S_{2n-1}])$.
\end{lemma}
\begin{proof}
We have the following intersection products in $H_2(X;\Z/2\Z)$:
$$[S_i]\cdot [S_j]=0 \ \mbox{if} \ |i-j|\ne 1,  \quad [S_i]\cdot
[S_{i+1}]=1, \quad [S_i]\cdot [E_2]=0  \ \mbox{if} \ i\ne 2,3,  \quad [S_i]\cdot
[E_2]=1 \ \mbox{if} \ i= 2,3,$$
$$ c_1(X)\cdot[E_2]= 1, \quad  c_1(X)^2=[B]\cdot c_1(X)= [F]\cdot
c_1(X)= [S_i]\cdot c_1(X)=0,$$
and
$$
 [F]^2=[F]\cdot [S_i]= [F]\cdot [E_2]= [B]\cdot [E_2]= [B]\cdot [S_i]=[B]^2=0,
\quad [E_2]^2=1,\quad [B]\cdot [F]=1. $$
In particular
$(c_1(X),[B],[F],[E_2],[S_2]\ldots,[S_{2n-1}])$ is a free family of
$H_2(X;\Z/2\Z)$, and hence is a basis
since $b_2(X; \Z/2\Z)=2n+2$.

From the intersection $[S_1]\cdot [S_2]=1$, we deduce that $[S_1]\ne 0$
in $H_2(X;\Z/2\Z)$, and Lemma \ref{lem:MV} implies that
$b_2(X\setminus
L;\Z/2\Z)=b_2(X;\Z/2\Z)-1$.
A basis of
 $H_2(X\setminus L;\Z/2\Z)$ is then given by
$$(c_1(X),[B],[F],[E_2],[S_3],\ldots,[S_{2n-1}]),$$
since its rank in $H_2(X;\Z/2\Z)$ is at most its rank in $H_2(X\setminus L;\Z/2\Z)$.
The classes
$c_1(X),[F]$, and $[S_i]$
 are $\tau$-invariant,
and
 we have:
$$\tau_*[B]=[B]+n[F]+c_1(X),\quad \mbox{and}\quad \tau_*[E_2]=[E_2]+F. $$
It follows that $(c_1(X),[F],[S_3],\ldots,[S_{2n-1}])$ is 
a basis of  
the subspace of $\tau$-invariant classes of $H_2(X\setminus
L;\Z/2\Z)$, and the lemma is proved.
\end{proof}

\subsection{Minimal real Del Pezzo surface of degree 2}\label{sec:quartic}
Let $Q$ be the real quartic in $\C P^2$ 
whose  real part together its position with respect to a
bitangent $H$
is depicted in Figure \ref{fig:quartic}a. We denote by $(X,\tau)$ 
the real double covering $\rho: X\to \C P^2$ ramified along $Q$, whose real
part consists of four spheres.
The real Picard group of $X$ is the free abelian group generated by
$c_1(X)$ (see \cite{Sil89,Kol97}).

There exists a $(-1)$-curve $E$
such that $\rho(E)=H$. 
Let $S_1,S_3,S_5,$ and $S_7$ by the four spheres of $\R X$.
By the rigid isotopy classification of real plane quartics, each pair
of real spheres is connected by a $\tau$-invariant 
vanishing Lagrangian sphere. Let
$S_2$ (resp. $S_4$,  $S_6$) such a sphere connecting $S_1$ and $S_3$ 
(resp.  $S_3$ and $S_5$, $S_5$ and $S_7$) as depicted in  Figure
\ref{fig:quartic}b.
 \begin{figure}[h]
\centering
\begin{tabular}{ccc}
\includegraphics[height=5cm,  angle=0]{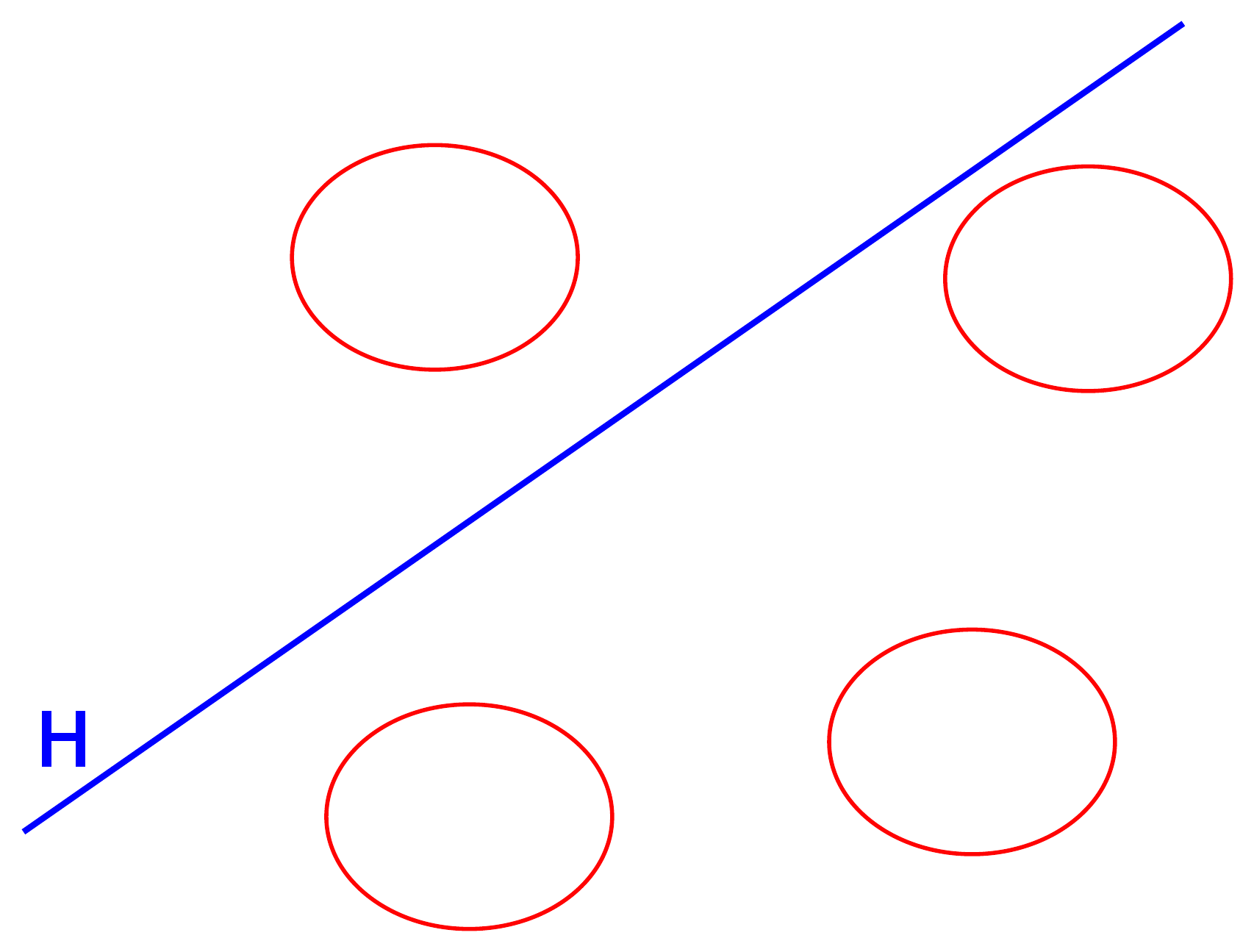}& \hspace{3ex} &
\includegraphics[height=5cm,  angle=0]{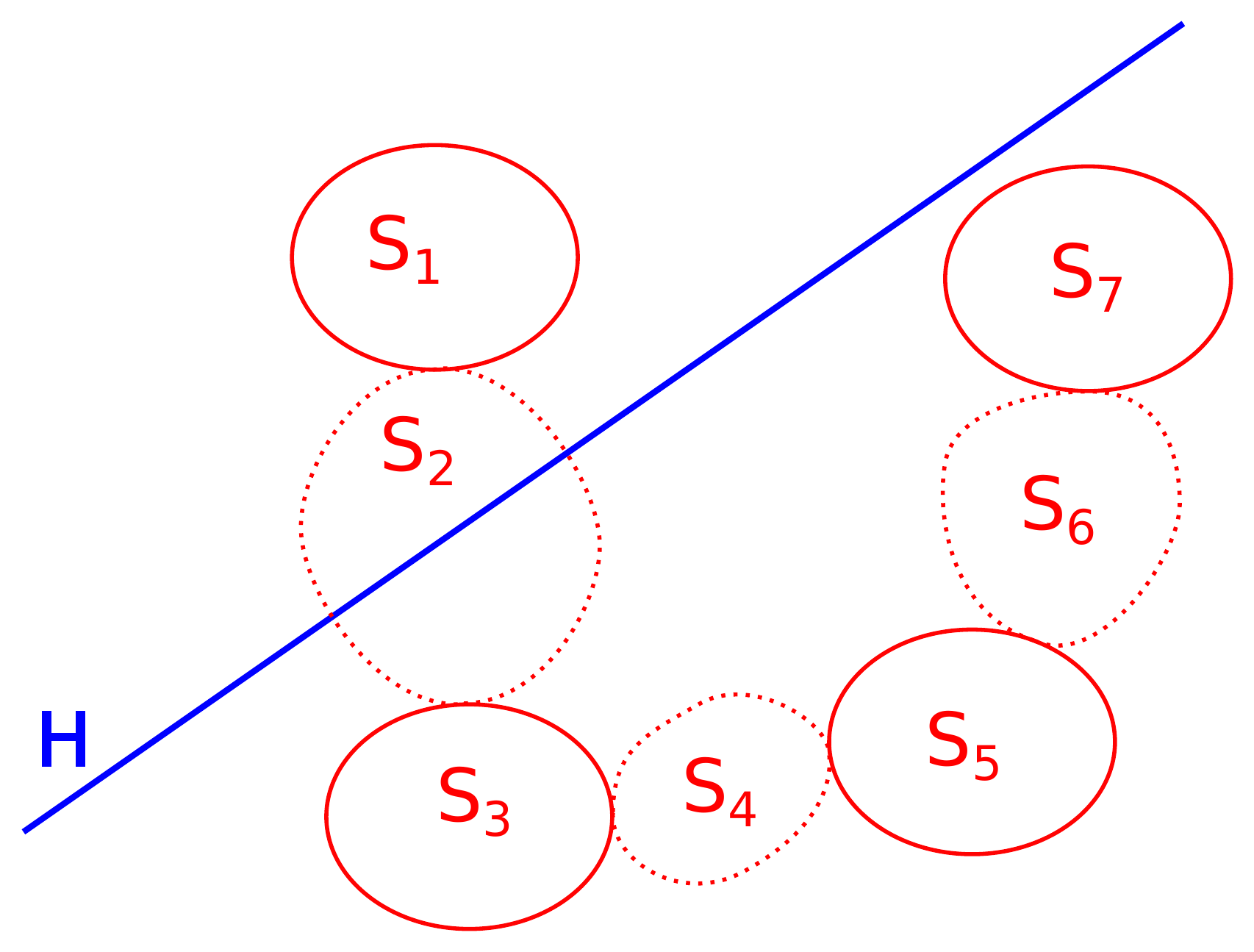}
\\ \\ a) && b)
\end{tabular}
\caption{Real vanishing cycles of a minimal real Del Pezzo surface of
  degree $2$}
\label{fig:quartic}
\end{figure}
\begin{lemma}\label{lem:homology quartic}
A basis of $\mathcal H(X_\R,S_1)$ is given by
$([S_3],\ldots,[S_{7}])$.
\end{lemma}
\begin{proof}
We have the following intersection products in $H_2(X;\Z/2\Z)$:
$$[S_i]\cdot [S_j]=0 \ \mbox{if} \ |i-j|\ne 1,  \quad [S_i]\cdot
[S_{i+1}]=1, \quad [S_i]\cdot [E]=0  \ \mbox{if} \ i\ne 2,  \quad [S_2]\cdot
[E]=1,$$ 
$$ c_1(X)\cdot[E]= 1, \quad  c_1(X)^2= [S_i]\cdot c_1(X)=0,\quad
 \mbox{and}\quad [E]^2=1.$$ 
In particular $(c_1(X),[E],[S_2],\ldots,[S_7])$ is a basis of
 $H_2(X;\Z/2\Z)$, and
 $(c_1(X),[E],[S_3],\ldots,[S_7])$  is a basis of
 $H_2(X\setminus L;\Z/2\Z)$.
The classes
$c_1(X)$ and $[S_i]$
 are $\tau$-invariant,
and
 $\tau_*[E]=c_1(X)+[E]$.
Hence $(c_1(X),[S_3],\ldots,[S_{2n-1}])$ is 
a basis of  
the subspace of $\tau$-invariant classes of $H_2(X\setminus
L;\Z/2\Z)$, and the lemma is proved.
\end{proof}

\subsection{Minimal real Del Pezzo surface of degree 1}\label{sec:cone}
Let $Q$ be the real cubic section of the quadratic cone $\Sigma$ in
$\C P^3$ whose  real part together with its position with respect to a
tritangent hyperplane section $H$ is
depicted in figure \ref{fig:cone}a. We denote by $(X,\tau)$ 
the real double covering $\rho: X\to \Sigma$ ramified along $Q$
 whose real
part consists of four spheres and a real projective plane.
The real Picard group of $X$ is the free abelian group generated by
$c_1(X)$ (see \cite{Sil89,Kol97}).
 \begin{figure}[h]
\centering
\begin{tabular}{ccc}
\includegraphics[height=5cm,  angle=0]{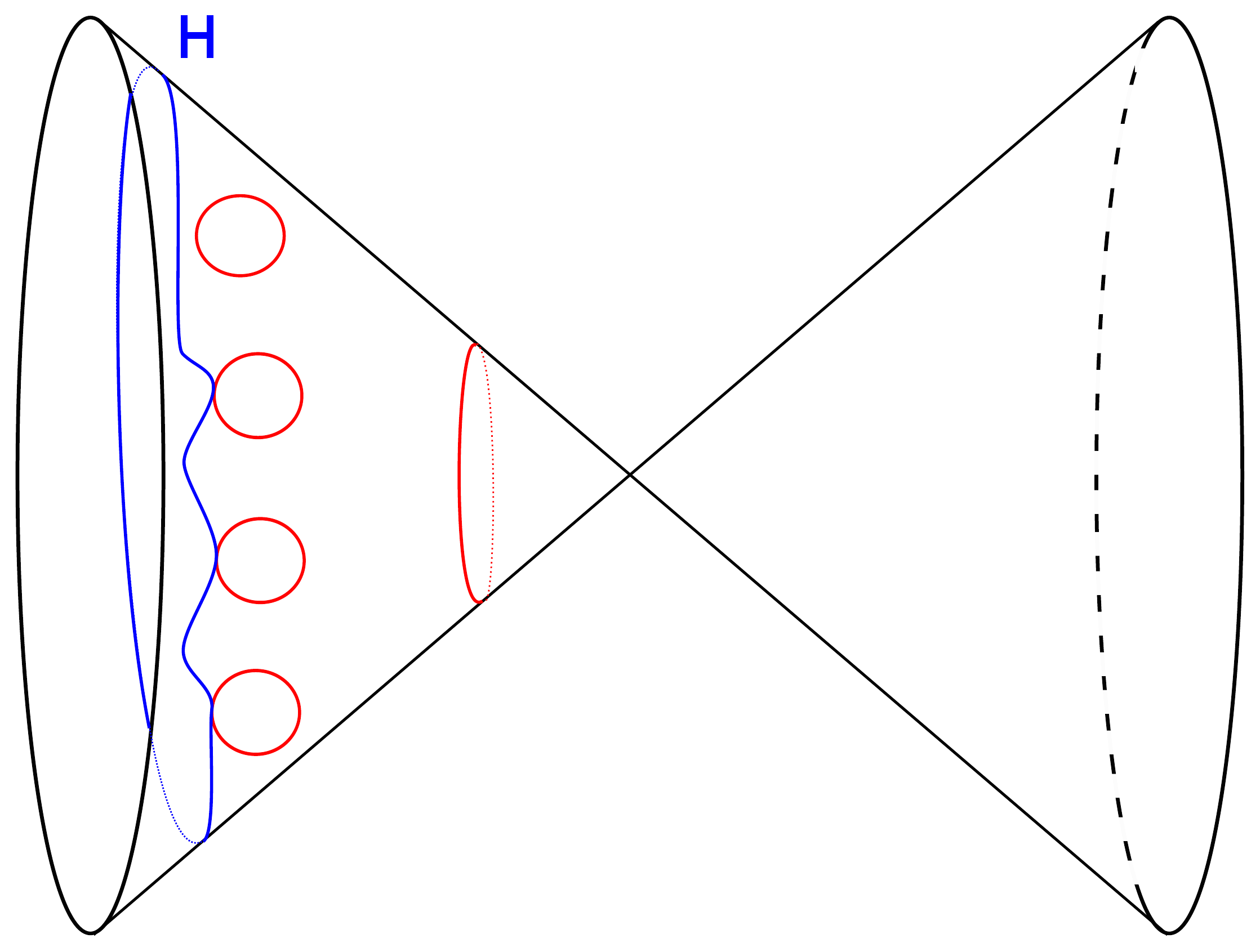}& \hspace{3ex} &
\includegraphics[height=5cm,  angle=0]{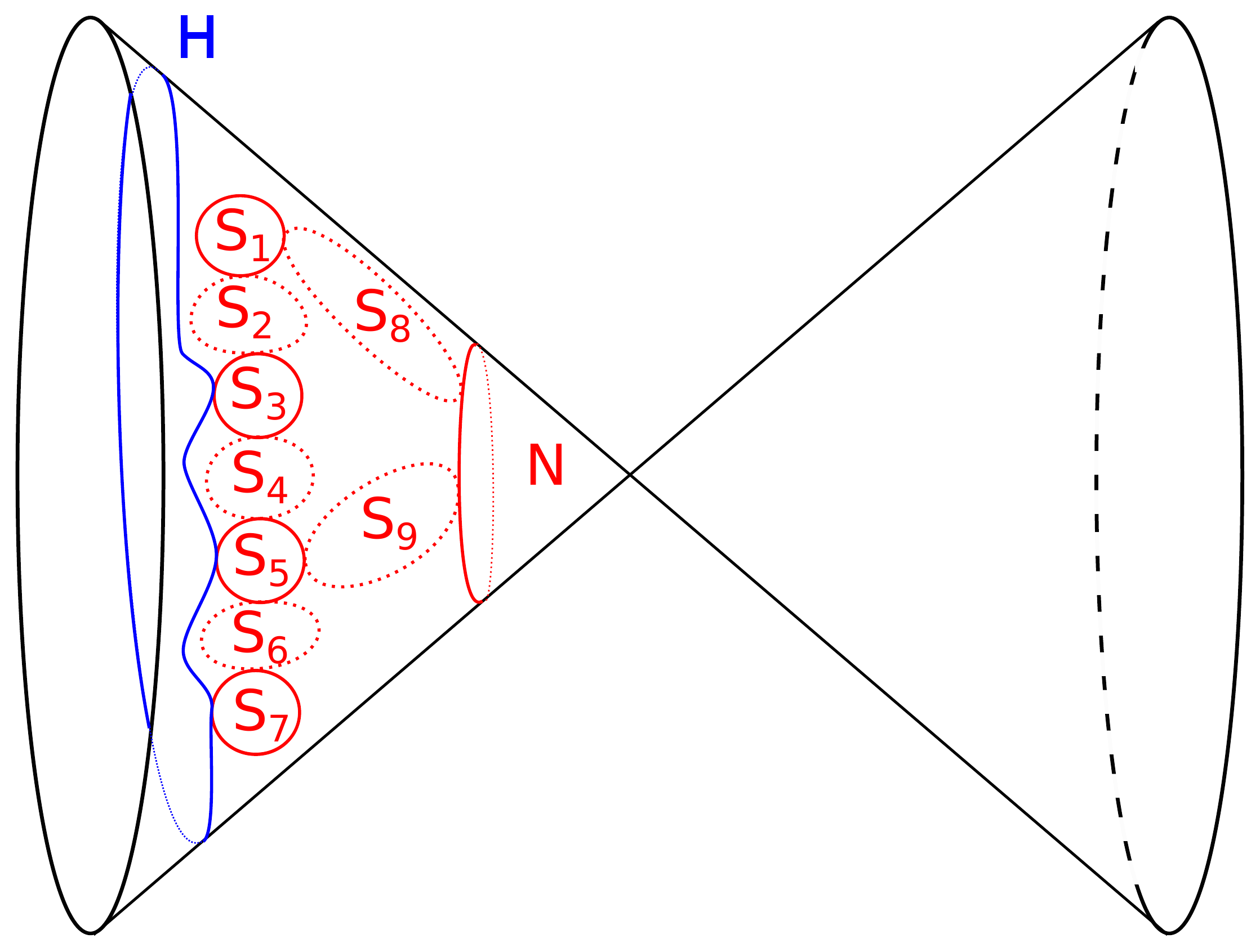}
\\ \\ a) && b)
\end{tabular}
\caption{Real vanishing cycles of a minimal real Del Pezzo surface of
  degree $1$}
\label{fig:cone}
\end{figure}

There exists a $(-1)$-curve $E$
such that $\rho(E)=H$. 
Let $S_1,S_3,S_5,S_7$ and $N$ be respectively
 the four spheres and the real projective plane  of $\R X$.
By the rigid isotopy classification of  real cubic sections of
$\Sigma$, there exist $\tau$-invariant vanishing Lagrangian spheres
$S_2,S_4,S_6,S_8,S_9$ as depicted in Figure \ref{fig:cone}b.
Note that $\tau$ acts trivially on $H_2(X;\Z/2\Z)$.
\begin{lemma}\label{lem:homology cone}
A basis of $\mathcal H(X_\R,S_1)$ is given by
$([S_3],\ldots,[S_{7}],[S_9],[N])$.

A basis of $\mathcal H(X_\R,S_7)$ is given by
$([S_1],\ldots,[S_{5}],[S_8],[N])$.

A basis of $\mathcal H(X_\R,N)$ is given by
$([S_1],\ldots,[S_{7}])$.
\end{lemma}
\begin{proof}
All intersection products of  $[S_i]$ with $[S_j]$, $[N]$, and $[E]$,  
 can be read
on  Figure \ref{fig:cone}b. The other intersection products in
$H_2(X;\Z/2\Z)$
are:
$$  c_1(X)^2=[E]^2=[N]^2=[E]\cdot c_1(X)=[N]\cdot c_1(X) =1,  \quad
 [N]\cdot [E]=0.$$
Hence $(c_1(X),[S_1],\ldots,[S_8])$ is a basis of
 $H_2(X;\Z/2\Z)$, and we have
$$[N]=c_1(C)+ [S_1]+ [S_3]+ [S_5]+ [S_7], \quad\mbox{and}\quad
[S_9]=[S_8]+[S_2]+[S_4].$$
The result about  $\mathcal H(X_\R,S_1)$ and  $\mathcal H(X_\R,S_7)$
follows immediately.

The generator of $Ker \ \iota$ can be represented by $B=\rho^{-1}(A)$,
where $A$ is a hyperplane section of $\Sigma$.
Hence
 $([B],[E],[S_1],\ldots,[S_7])$  is a basis of
 $H_2(X\setminus N;\Z/2\Z)$.
Since $[E]+\tau_*[E]=B$, we get that
 $([B],[S_1],\ldots,[S_{7}])$ is 
a basis of  
the subspace of $\tau$-invariant classes of $H_2(X\setminus
N;\Z/2\Z)$, and the lemma is proved.
\end{proof}

\subsection{Proof of Theorem \ref{thm:real algebraic}}\label{sec:proof}

Lemmas \ref{lem:H vanish}, \ref{lem:homology conic}, \ref{lem:homology
  quartic}, and 
\ref{lem:homology cone} and Theorem \ref{thm:chain spheres} provide a
proof of Theorem \ref{thm:real algebraic}
in the case of minimal real algebraic rational surfaces, and when
$F=[\R X\setminus L]$.
To end the proof, we start with the following remark:
if $(\widetilde X,\widetilde \tau)$ is a blow up of $(X,\tau)$ at a real
point or at a pair of $\tau$-conjugated points, and if $\widetilde L$
is the component of $\R \widetilde X$ corresponding to $L$, 
then there is a
natural injective group homomorphism $\phi:\mathcal H(X_\R,L)\to
\mathcal H(\widetilde X_\R, \widetilde L)$. 
Theorem \ref{thm:real algebraic} now
 follows immediately from next proposition.

\begin{prop}\label{prop:isomorphism H}
The map $\phi$ is an isomorphism.
\end{prop}
\begin{proof}
This is clearly true when  $(\widetilde X,\widetilde \tau)$ is a blow
up of $(X,\tau)$ at a real
point in $\R X\setminus L$ or at a pair of $\tau$-conjugated points.
Hence let us now assume that $(\widetilde X,\widetilde \tau)$ is a blow
up of $(X,\tau)$ at a point $p\in L$.
Since $[\widetilde L]\ne 0$ in $H_2(X;\Z/2\Z)$, by Lemma \ref{lem:MV} we have
$$b_2(\widetilde X\setminus \widetilde L;\Z/2\Z)=b_2(X\setminus L;\Z/2\Z)+1$$
if $(X,\tau)=(\C P^1\times \C P^1,\tau_{hy})$, and
$$b_2(\widetilde X\setminus \widetilde L;\Z/2\Z)=b_2(X\setminus L;\Z/2\Z)+2$$
otherwise.
In both cases, an extra generator of $H_2(\widetilde X\setminus
\widetilde L;\Z/2\Z)$  is given by the extra generator of
$H_1(\widetilde L;\Z/2\Z)$. In particular this proves the proposition
in the case $(X,\tau)=(\C P^1\times \C P^1,\tau_{hy})$, and implies
$$\dim  \mathcal H(\widetilde X_\R, \widetilde L)\le \dim \mathcal H(X_\R,L)+1$$
otherwise.
In this latter case, from the classification of minimal real algebraic surfaces
up to deformation,
 we may assume that there exists an algebraic curve $C$ in $X$
such that $C\cap L=\{p\}$ and that this intersection is transverse.
Hence the strict transform 
of $C$ in $\widetilde X$ is 
a second extra generator of $H_2(\widetilde X\setminus \widetilde L;\Z/2\Z)$, which is
either not $\tau$-invariant or mapped to $0$ in 
$ \mathcal H(\widetilde X_\R, \widetilde L)$.
\end{proof}

\bibliographystyle {alpha}
\bibliography {Biblio.bib}

\end{document}